\documentclass[11pt,a4paper,twoside,english,noabbrev]{article}
\usepackage[headcolor=OliveGreen]{article-style}
\hypersetup{bookmarksopenlevel=2}

\makeatletter
\renewcommand*\env@matrix[1][c]{\hskip -\arraycolsep
  \let\@ifnextchar\new@ifnextchar
  \array{*\c@MaxMatrixCols #1}}
\makeatother

\makeatletter
\def\hdots@for#1#2{\multicolumn{#2}c%
  {\m@th\dotsspace@1.5mu\mkern-#1\dotsspace@
   \xleaders\hbox{$\m@th\mkern#1\dotsspace@\cdot\mkern#1\dotsspace@$}%
           \hskip\z@\@plus 1filll
   \mkern-#1\dotsspace@}%
   }
\makeatother

\RequirePackage{etoolbox}
\makeatletter
\patchcmd{\@adjustvertspacing}
  {\jot\baselineskip \divide\jot 4}
  {\jot=.5\baselineskip}
  {}{}
\makeatother



\newcommand*{\diffdchar}{\textnormal{d}}
\newcommand*{\dee}{\mathop{\diffdchar\!}}

\DeclareMathOperator{\support}{supp}

\DeclarePairedDelimiterX{\norm}[1]{\lVert}{\rVert}{\ifblank{#1}{\,\cdot\,}{#1}}
\DeclarePairedDelimiterX{\seminorm}[1]{\lvert}{\rvert}{\ifblank{#1}{\,\cdot\,}{#1}}
\DeclarePairedDelimiterX{\abs}[1]{\lvert}{\rvert}{\ifblank{#1}{\,\cdot\,}{#1}}
\DeclarePairedDelimiterX{\innerproduct}[2]{\langle}{\rangle}{\ifblank{#1}{\,\cdot\,}{#1},\ifblank{#2}{\,\cdot\,}{#2}}
\DeclarePairedDelimiterX{\dualitypairing}[2]{\langle}{\rangle}{\ifblank{#1}{\,\cdot\,}{#1},\ifblank{#2}{\,\cdot\,}{#2}}
\DeclarePairedDelimiterXPP{\distance}[2]{\operatorname{dist}}{(}{)}{}{\ifblank{#1}{\,\cdot\,}{#1},\ifblank{#2}{\,\cdot\,}{#2}}

\providecommand\given{}
 \newcommand\SetSymbol[1][]{%
      \nonscript\: :
      \allowbreak
      \nonscript\:
      \mathopen{}}
   \DeclarePairedDelimiterX\Set[1]\{\}{%
      \renewcommand\given{\SetSymbol[\delimsize]}
      #1
   }

\DeclareMathOperator*{\vectorspan}{span}

\DeclareMathOperator*{\identityoperator}{id}

\renewcommand\Re{\operatorname{Re}}

\makeatletter
\newcommand{\biggg}{\bBigg@{2.6}}
\newcommand{\Biggg}{\bBigg@{2.87}}
\newcommand{\vast}{\bBigg@{3}}
\newcommand{\Vast}{\bBigg@{5}}
\makeatother

\renewcommand\leq\leqslant
\renewcommand\geq\geqslant

\DeclareRobustCommand{\A}[0]{\mathbb{A}}
\DeclareRobustCommand{\B}[0]{\mathbb{B}}
\DeclareRobustCommand{\C}[0]{\mathbb{C}}
\DeclareRobustCommand{\D}[0]{\mathbb{D}}
\DeclareRobustCommand{\E}[0]{\mathbb{E}}
\DeclareRobustCommand{\F}[0]{\mathbb{F}}
\DeclareRobustCommand{\G}[0]{\mathbb{G}}
\DeclareRobustCommand{\H}[0]{\mathbb{H}}
\DeclareRobustCommand{\I}[0]{\mathbb{I}}
\DeclareRobustCommand{\J}[0]{\mathbb{J}}
\DeclareRobustCommand{\K}[0]{\mathbb{K}}
\DeclareRobustCommand{\L}[0]{\mathbb{L}}
\DeclareRobustCommand{\M}[0]{\mathbb{M}}
\DeclareRobustCommand{\N}[0]{\mathbb{N}}
\DeclareRobustCommand{\O}[0]{\mathbb{O}}
\DeclareRobustCommand{\P}[0]{\mathbb{P}}
\DeclareRobustCommand{\Q}[0]{\mathbb{Q}}
\DeclareRobustCommand{\R}[0]{\mathbb{R}}
\DeclareRobustCommand{\S}[0]{\mathbb{S}}
\DeclareRobustCommand{\T}[0]{\mathbb{T}}
\DeclareRobustCommand{\U}[0]{\mathbb{U}}
\DeclareRobustCommand{\V}[0]{\mathbb{V}}
\DeclareRobustCommand{\W}[0]{\mathbb{W}}
\DeclareRobustCommand{\X}[0]{\mathbb{X}}
\DeclareRobustCommand{\Y}[0]{\mathbb{Y}}
\DeclareRobustCommand{\Z}[0]{\mathbb{Z}}

\makeatletter
\let\save@mathaccent\mathaccent
\newcommand*\if@single[3]{%
  \setbox0\hbox{${\mathaccent"0362{#1}}^H$}%
  \setbox2\hbox{${\mathaccent"0362{\kern0pt#1}}^H$}%
  \ifdim\ht0=\ht2 #3\else #2\fi
  }
\newcommand*\rel@kern[1]{\kern#1\dimexpr\macc@kerna}
\newcommand*\widebar[1]{\@ifnextchar^{{\wide@bar{#1}{0}}}{\wide@bar{#1}{1}}}
\newcommand*\wide@bar[2]{\if@single{#1}{\wide@bar@{#1}{#2}{1}}{\wide@bar@{#1}{#2}{2}}}
\newcommand*\wide@bar@[3]{%
  \begingroup
  \def\mathaccent##1##2{%
    \let\mathaccent\save@mathaccent
    \if#32 \let\macc@nucleus\first@char \fi
    \setbox\z@\hbox{$\macc@style{\macc@nucleus}_{}$}%
    \setbox\tw@\hbox{$\macc@style{\macc@nucleus}{}_{}$}%
    \dimen@\wd\tw@
    \advance\dimen@-\wd\z@
    \divide\dimen@ 3
    \@tempdima\wd\tw@
    \advance\@tempdima-\scriptspace
    \divide\@tempdima 10
    \advance\dimen@-\@tempdima
    \ifdim\dimen@>\z@ \dimen@0pt\fi
    \rel@kern{0.6}\kern-\dimen@
    \if#31
      \overline{\rel@kern{-0.6}\kern\dimen@\macc@nucleus\rel@kern{0.4}\kern\dimen@}%
      \advance\dimen@0.4\dimexpr\macc@kerna
      \let\final@kern#2%
      \ifdim\dimen@<\z@ \let\final@kern1\fi
      \if\final@kern1 \kern-\dimen@\fi
    \else
      \overline{\rel@kern{-0.6}\kern\dimen@#1}%
    \fi
  }%
  \macc@depth\@ne
  \let\math@bgroup\@empty \let\math@egroup\macc@set@skewchar
  \mathsurround\z@ \frozen@everymath{\mathgroup\macc@group\relax}%
  \macc@set@skewchar\relax
  \let\mathaccentV\macc@nested@a
  \if#31
    \macc@nested@a\relax111{#1}%
  \else
    \def\gobble@till@marker##1\endmarker{}%
    \futurelet\first@char\gobble@till@marker#1\endmarker
    \ifcat\noexpand\first@char A\else
      \def\first@char{}%
    \fi
    \macc@nested@a\relax111{\first@char}%
  \fi
  \endgroup
}
\makeatother
\usepackage[only,llbracket,rrbracket]{stmaryrd}
\DeclarePairedDelimiterX{\cosineexpansion}[1]{\llbracket}{\rrbracket}{\ifblank{#1}{\,\cdot\,}{#1}}
\usepackage{subcaption}
\usepackage{floatrow,array,multirow}
\usepackage{tasks}
\RequirePackage{tikz}
\RequirePackage{pgfplots}
\pgfplotsset{compat=newest}
\newlength\figureheight 
\newlength\figurewidth
\color{artcolortext} 

\begin{document}
\title{Periodic Hölder waves in a class of\\ negative-order dispersive equations}
\author{Fredrik Hildrum\footnote{Corresponding author.} \,\orcid{0000-0002-9905-1670}\email{arXiv@fredrik.hildrum.net} \and Jun Xue\,\orcid{0000-0002-1888-0208}\email{jun.xue@ntnu.no}}
\address{Department of Mathematical Sciences,\\ NTNU -- Norwegian University of Science and Technology,\\ 7491 Trondheim, Norway}
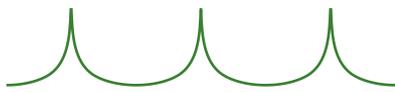
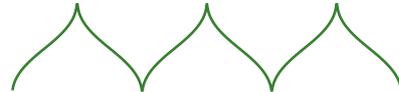
\titlerunning{Hölder waves in negative-order dispersive equations}
\authorrunning{F.~Hildrum and J.~Xue}
\date{\today}
\maketitle

\titlegraphic{%
\vspace*{-0.8em}%
\centering%
\setlength\figureheight{.1\textwidth}%
\setlength\figurewidth{.4\textwidth}%
\input{tikz/title-waves-abs.tikz}%
\hspace{.16\textwidth}%
\input{tikz/title-waves-sgn.tikz}%
}

\begin{abstract}
We prove the existence of highest, cusped, periodic travelling-wave solutions with exact and optimal \({ \alpha }\)-Hölder continuity in a class of fractional negative-order dispersive equations of the form
\begin{equation*}
u_t + (\abs{ \textnormal{D} }^{- \alpha} u + n(u) )_x = 0
\end{equation*}
for every \({ \alpha \in (0, 1) }\) with homogeneous Fourier multiplier~\({\abs{ \textnormal{D} }^{ - \alpha} }\). We tackle nonlinearities \({ n(u) }\) of the type \({ \abs{ u }^p }\) or \({ u \abs{ u }^{p - 1} }\) for all real~\({ p > 1 }\), and show that when \({ n }\) is odd, the waves also feature antisymmetry and thus contain inverted cusps. Tools involve detailed pointwise estimates in tandem with analytic global bifurcation, where we resolve the issue with nonsmooth~\({ n }\) by means of regularisation. We believe that both the construction of highest antisymmetric waves and the regularisation of nonsmooth terms to an analytic bifurcation setting are new in this context, with direct applicability also to generalised versions of the Whitham, the Burgers--Poisson, the Burgers--Hilbert, the Degasperis--Procesi, the reduced Ostrovsky, and the bidirectional Whitham equations.
\end{abstract}


\section{Introduction} \label{sec:intro}

\subsection{Main result}

In this paper, we shall be concerned with singular periodic travelling-wave solutions to a class of nonlinear and dispersive evolution equations of the form
\begin{equation} \label{eq:evolution}
u_t + (\abs{ \textnormal{D} }^{- \alpha}u + n(u))_x = 0.
\end{equation}
This family may be viewed as a kind of generalised fractional Korteweg--de~Vries (KdV) equations of negative-order, where we refer to \autocite{BenBonMah1972q} for a classical description of nonlocal variants of the KdV equation in the mathematical modelling of long-wave phenomena. The dispersive properties occur in the homogeneous negative-order (spatial) Fourier multiplier~\({ \abs{ \textnormal{D} }^{- \alpha} }\) for \({ \alpha \in (0, 1) }\) defined by
\begin{equation*}
\mathscr{F}( \abs{ \textnormal{D} }^{- \alpha} u)(\xi) \coloneqq \abs{ \xi }^{- \alpha} \, \widehat{  u  }(\xi),
\end{equation*}
with \({ \textnormal{D} \coloneqq - \textnormal{i} \partial_x }\), whereas the nonlinear effects originate from either of the generally nonsmooth nonlinearities
\begin{subequations} \label{eq:n}
    \begin{empheq}[left={n(x) \coloneqq \empheqlbrace\,},right={\,\empheqrbrace\qquad\text{with \({ p > 1 }\) real.}}]{align}
      & \abs{ x }^{p} \text{ or} \label{eq:n-abs} \tag{\({ \theparentequation_{ \textnormal{abs}} }\)}\\
      & x \abs{ x }^{p - 1} \label{eq:n-sgn} \tag{\({ \theparentequation_{ \textnormal{sgn}} }\)}
    \end{empheq}
\end{subequations}
Our main contributions are to
\begin{enumerate}[ref=\roman*),itemsep=.5\parskip]
\item 
prove the existence of highest, exactly \({ \alpha }\)-Hölder continuous periodic steady solutions of the negative-order dispersive family~\eqref{eq:evolution} for all~\({ \alpha \in (0, 1) }\) on the torus~\({ \T \coloneqq \R / 2 \uppi \Z }\), and
\item
initiate a study of nonsmooth nonlinearities and antisymmetric features in the large-amplitude theory for negative-order dispersive evolution equations.
\end{enumerate}
Precisely, we obtain the following result, with corresponding numerical illustrations in \cref{fig:waves}.

\begin{theorem}[Existence] \label{thm:existence}
Let \({ \alpha \in (0, 1) }\) and \({ p > 1 }\) be real. Then there exists a nontrivial
periodic travelling-wave solution~\({ \varphi }\) of~\eqref{eq:evolution} with positive speed \({ c < \frac{ p }{ p - 1 } \norm{ \mathscr{F}^{-1}( \abs{  }^{ - \alpha}) }_{ \textnormal{L}^1(\T) }  }\). The solution is even (about~\({ 2 \uppi \Z }\)), has zero mean, and satisfies
\begin{equation*}
\max \varphi = \varphi(0) = \mu \qquad \text{and} \qquad \varphi \in \textnormal{C}^\alpha(\T),
\end{equation*}
where \({ \mu \coloneqq (c/p)^{1/(p - 1)} }\). It is also smooth (except possibly at the point where it vanishes) and strictly increasing on~\({ (-\uppi, 0) }\) and exactly \({ \alpha }\)-Hölder continuous at \({ x \in 2 \uppi \Z }\), that is,
\begin{equation*}
\mu - \varphi(x) \eqsim \abs{ x - 2\uppi \ell }^{ \alpha }
\end{equation*}
uniformly around~\({ 2\uppi \ell }\) for \({ \ell \in \Z }\).

One has that \({ \varphi }\) is smooth around~\({ - \uppi }\) in case~\eqref{eq:n-abs}, while \({ \varphi }\) is antisymmetric about~\({ - \frac{ \uppi }{ 2 } }\) in case~\eqref{eq:n-sgn} and therefore also exactly \({ \alpha }\)-Hölder continuous at~\({ \uppi \Z }\) with \({ \min \varphi = \varphi(- \uppi) = - \mu  }\).
\end{theorem}

\begin{remark}
\({ A \eqsim B }\) is short for \({ A \lesssim B \lesssim A }\), where \({ A \lesssim B }\) symbolises that \({ A \leq \lambda B }\) for a constant~\({ \lambda > 0 }\). We say that \({ A(x) \lesssim B(x) }\) (etc.\@) holds \emph{uniformly} over a region if \({ \lambda }\) does not depend on~\({ x }\) there.
\end{remark}

\begin{figure}[h!]
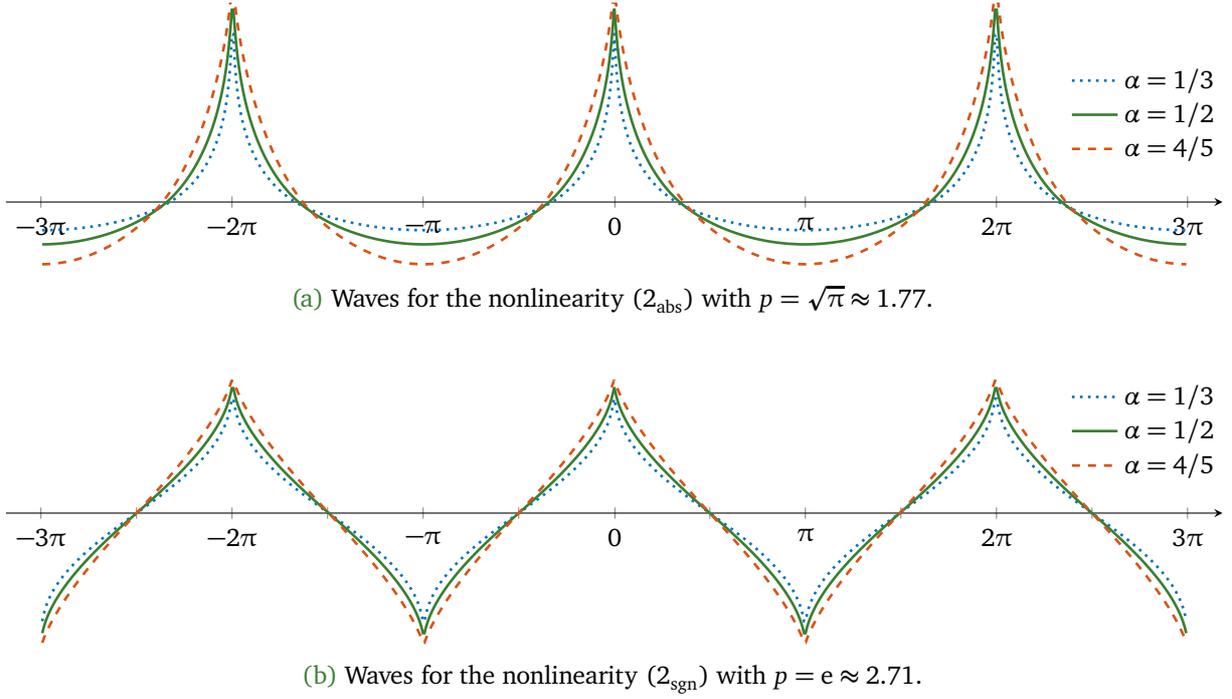

	\vspace*{-.5em}%
    \centering%
    \begin{subfigure}[b]{\textwidth}%
    \centering%
\setlength\figureheight{.22\textwidth}%
\setlength\figurewidth{\textwidth}%
\input{tikz/waves-abs.tikz}%
        \caption{Waves for the nonlinearity~\eqref{eq:n-abs} with~\({ p = \sqrt{ \uppi } \approx 1.77 }\).}
       \label{fig:waves-abs}
    \end{subfigure}
    \\[2em] %
    \begin{subfigure}[b]{\textwidth}%
\centering%
\setlength\figureheight{.22\textwidth}%
\setlength\figurewidth{\textwidth}%
\input{tikz/waves-sgn.tikz}%
        \caption{Waves for the nonlinearity~\eqref{eq:n-sgn} with~\({ p = \textnormal{e} \approx 2.71 }\).}
        \label{fig:waves-sgn}
    \end{subfigure}
    \caption{Numerical approximations of large-amplitude, \({ \textnormal{C}^\alpha }\)-regular periodic waves for various~\({ \alpha }\)'s and nonlinearities~\eqref{eq:n} in \cref{thm:existence} using the Spec\-TraVVave software~\autocite{KalMolVer2017a} for bifurcation in nonlinear dispersive evolution equations.}
    \label{fig:waves}
\end{figure}

\subsection{Background}

Full-dispersion nonlinear evolution equations such as~\eqref{eq:evolution} have seen a keen interest in the recent years as nonlocal improvements of classical local equations. In particular, surface-wave models in shallow water of this class with various dispersive operators approximate the full water-wave equations~\autocite{Eme2021a,Eme2021b} and capture singular features not found in their local counterparts. \Cref{eq:evolution} with \({ n(u) = u^2 }\) may be seen as a dispersive perturbation of Burgers' equation and~\autocite{Hur2012g} outlines how this case for \({ \alpha = \frac{ 1 }{ 2 } }\) is perhaps the simplest model incorporating the linear behaviour and characteristic nonlinearity of the water-waves problem. Both~\autocite{Hur2012g} \({\bigl(\text{for } \alpha = \frac{ 1 }{ 2 } \bigr) }\) and~\autocite[Theorem~3.1]{CasCorGan2010a} (for~\({ \alpha \in (0, 1) }\)) prove that solutions of~\eqref{eq:evolution} blow up in finite time for certain initial data. The resulting singularities occur in at least two ways: wave breaking~\autocite{HurTao2014}, in which the spatial derivative of a bounded solution of~\eqref{eq:evolution} blows up, or sharp crests in travelling-wave solutions -- reminiscent of the highest Stokes' wave~\autocite{TolBen1978a} -- which is the subject of this paper. We refer also to~\autocite{LinPilSau2014a,KleLinPilSau2018a,EhrWan2019j,KleSauWan2022c,Ria2021u} for other results concerning singularities, well-posedness, persistence, and existence time of solutions.

Classical Fourier analysis shows that \({ \abs{ \textnormal{D} }^{- \alpha} }\) constitutes a singular convolution operator on~\({ \R }\) with kernel~\({ \abs{  }^{ \alpha - 1 } }\) and describes an eigenfunction of~\({ \mathscr{F} }\) when~\({ \alpha = \frac{ 1 }{ 2 } }\). Related to this case is the recent work by Ehrnström~\& Wahlén~\autocite{EhrWah2019a} on the Whitham equation~\autocite{Whitham1967}, being a shallow-water model of type~\eqref{eq:evolution} with inhomogeneous dispersion \({ \sqrt{ \tanh (\textnormal{D}) / \textnormal{D} } }\) and~\({ n(u) = u^2 }\). Its corresponding symbol behaves as that of the KdV~equation for small frequencies and decays like \({ \abs{ \xi }^{- \frac{ 1 }{ 2 }} }\) as \({ \abs{ \xi } \to \infty }\), for which the kernel may be written as \({ \abs{ x }^{- \frac{ 1 }{ 2 }} }\) plus a regular term. The existence of a highest, cusped steady solution whose behaviour at the crest is like~\({ 1 - \abs{ x }^{ \frac{ 1 }{ 2 } }  }\) modulo constants was conjectured by Whitham~\autocite[p.~479]{Whitham1974}, and the authors of~\autocite{EhrWah2019a} found this exactly \({ \textnormal{C}^{ \frac{ 1 }{ 2 }} }\)~wave on~\({ \T }\) based on properties of the kernel, precise regularity estimates, and global bifurcation theory, building on preliminary analysis from~\autocite{EhrnstromKalisch2009,EhrnstromKalisch2013}. New work~\autocite{TruWahWhe2021a} also proves the existence of an extreme \({ \textnormal{C}^{ \frac{ 1 }{ 2 }} }\) solitary-wave solution by means of nonlocal center-manifold theory for the global bifurcation.

\begin{table}[h!]
\vspace*{-.3em}%
\floatbox[{\capbeside\thisfloatsetup{capbesideposition={right,center},capbesidesep=quad}}]{table}[\FBwidth]
{\caption{Exact and global regularity of extreme periodic waves in negative-order equations with inhomogeneous or homogeneous dispersion. This paper also treats nonsmooth nonlinearities~\eqref{eq:n} for any real order \({ p > 1 }\), with applicability to the other works (that considered smooth~\({ n(u) = u^2 }\)).}\label{tab:overview}}
{%
\renewcommand{\arraystretch}{1.3}
\centering
\begin{tabularx}{.69\textwidth}{m{2.4cm}*{3}{>{\centering\arraybackslash}X}}
\toprule
\multirow{3.25}{*}{\hspace{.2cm}\parbox{2cm}{\centering \textsc{Dispersive operator}}} & \multicolumn{3}{c}{\textsc{Negative order and regularity}} \\
\cmidrule(lr){2-4}
& \({ \alpha \in (0, 1) }\) & \({ \alpha = 1 }\) & \({ \alpha > 1 }\) \\
& \({ \alpha }\)-Hölder & \({ \log }\)-Lipschitz & Lipschitz \\
\midrule
\centering\textit{Inhomogeneous:} \({ (\tanh(\textnormal{D})/ \textnormal{D})^{\alpha} }\) or~\({ (1 + \abs{ \textnormal{D} }^2)^{-\textstyle \frac{ \alpha }{ 2 }} }\) & \autocite{EhrWah2019a,EhrMaeVar2022a,Afr2021a,Ork2022a} & \autocite{EhrJohCla2019a,EhrMaeVar2022a} & \autocite{Le2021a} \\[2em]
\centering\parbox{2.4cm}{\centering\textit{Homogeneous:} \({ \abs{ \textnormal{D} }^{- \alpha} }\)} & \textcolor{artcolor}{This paper} & \autocite{EhrMaeVar2022a} & \autocite{BruDha2021a} \\[1.5ex]
\bottomrule
\end{tabularx}%
}\vspace*{-.6em}
\end{table}

Inspired by the results for the Whitham equation, there has been a series of papers concerned with exact and global regularity of extreme periodic waves in similar negative-order equations with prototypical inhomogeneous or homogeneous dispersion. As shown in \cref{tab:overview}, one obtains \({ \textnormal{C}^\alpha}\)~waves for~\({ \alpha \in (0, 1) }\), noting that~\autocite{EhrWah2019a} consider \({ (\tanh(\textnormal{D})/ \textnormal{D})^{\alpha} }\) for~\({ \alpha = \frac{ 1 }{ 2 } }\) and~\autocite{Afr2021a} the same dispersion for~\({ \alpha \in (0, 1) }\), and that~\autocite{Ork2022a} studies \({ (1 + \abs{ \textnormal{D} }^2)^{- \alpha / 2 } }\) for~\({ \alpha \in (0, 1) }\). When \({ \alpha = 1 }\), the extreme waves turn out to be \({ \log }\)-Lipschitz~\autocite{EhrJohCla2019a,EhrMaeVar2022a} in the sense that the behaviour at the crests is like \({ 1 - \abs{ x \log \abs{ x } } }\). In~\autocite{EhrMaeVar2022a}, they even provide exact asymptotics by new techniques, with applicability also to a subregime of~\({ \alpha \in (0, 1) }\) including the Whitham equation. We note that global existence of weak solutions are guaranteed by~\autocite{BreNgu2014s} in the homogeneous case of \({ \alpha = 1 }\) known as the Burgers–Hilbert equation. Finally, the waves are all Lipschitz~\autocite{BruDha2021a,Le2021a} when~\({ \alpha > 1 }\), and one naturally conjectures that the Lipschitz angles vanish as~\({ \alpha \searrow 1 }\). See further~\autocite{GeyPel2019a} for uniqueness and instability of the highest wave when~\({ \alpha = 2 }\), corresponding to the reduced Ostrovsky equation, and also~\autocite{Arn2019b} for the existence of extreme Lipschitz waves for Degasperis--Procesi equation.

\subsection{Contributions}

As promised and illustrated in \cref{tab:overview}, we complete the regularity picture for extreme periodic waves in the given negative-order class of dispersive equations. The regularity analysis emerges from the overall structure of~\autocite{EhrWah2019a} with the following key differences:
\begin{enumerate}[ref=\roman*),itemsep=.5\parskip]
\item
Whereas~\autocite{EhrWah2019a,BruDha2021a,Afr2021a} obtain monotonicity properties of the kernels based on a general characterisation of completely monotone functions or sequences, we establish monotonicity of the singular kernel
\begin{equation*}
K_\alpha(x) \coloneqq \mathscr{F}^{-1}( \abs{  }^{ - \alpha } )(x)
\end{equation*}
on~\({ \T }\) by computing an explicit integral representation valid for all~\({ \alpha > 0 }\), and use the Poisson summation formula to derive its precise singular behaviour (\({ \eqsim \abs{ x }^{ \alpha - 1} }\) as~\({ \abs{ x } \to 0 }\)) from the situation on~\({ \R }\).
\item
Since \({ K_\alpha }\) has only algebraic but not exponential decay (unlike the kernels in~\autocite{EhrWah2019a,Afr2021a,Ork2022a}), extra care must be applied to the finite-difference estimates for \({ \abs{ \textnormal{D} }^{ - \alpha}u }\) when~\({ \alpha }\) is arbitrarily close to~\({ 1 }\). In fact, we must exercise order-optimal estimates in order for the integrals to converge.
\item
We treat a class of nonlinearities, including sign-dependent ones~\eqref{eq:n-sgn}, in the regularity estimates, which amongst others requires the use of suitable properties of composition operators on Hölder spaces.
\end{enumerate}

The study of nonsmooth nonlinearities -- with both slow~(\({ p \gtrapprox 1 }\)) and arbitrary (polynomially) fast growth in~\eqref{eq:n} -- also poses new challenges since analytic bifurcation theory cannot be applied directly. We resolve this issue by analytically regularising the nonlinearities and proving that important features related to wave regularity and speed hold uniformly as the regularisation vanishes. The approach is strikingly simple (see~\eqref{eq:n-reg}).

In the special case of smooth \({ n(x) = x^p }\) with \({ 2 \leq p \in \N }\), we also compute in \cref{thm:bifurcation-formulas-smooth} local bifurcation formulas for all~\({ p }\), which may be of independent interest. We are also able to deduce the overall structure of the bifurcation formulas along the entire local bifurcation curve when \({ p }\) is odd. As for the case of general sign-dependent nonlinearities~\eqref{eq:n-sgn}, we establish that the highest waves exhibit antisymmetry and thus also contain an inverted cusp at the troughs, as illustrated in~\cref{fig:waves}. This construction appears to be completely new in the context of large-amplitude singular waves and sheds light on underlying symmetry principles. 

With appropriate modifications, these results are also transferable to other nonlocal dispersive equations. In particular, one may obtain such \enquote{doubly-cusped} periodic solutions (with zero mean) in the full scale of equations in \cref{tab:overview} with generalised nonlinearities of type~\eqref{eq:n}. Specifically, consider the evolution equation
\begin{equation} \label{eq:general-evolution}
u_t + (\mathcal{L}_\alpha u + n(u))_x = 0,
\end{equation}
where \({ \mathcal{L}_\alpha }\) is any of the dispersive operators
\begin{equation*}
(\tanh(\textnormal{D})/ \textnormal{D})^{\alpha}, \qquad \bigl(1 + \abs{ \textnormal{D} }^2 \bigr)^{-\textstyle \frac{ \alpha }{ 2 }} \qquad \text{or} \qquad \abs{ \textnormal{D} }^{ -\alpha }
\end{equation*}
for \({ \alpha \in (0, \infty) }\) and \({ n }\) is as in~\eqref{eq:n}. By readily adapting the regularity estimates in \autocite{EhrWah2019a,EhrJohCla2019a,Le2021a,BruDha2021a,Afr2021a,Ork2022a,EhrMaeVar2022a} with the estimates for general nonlinearities considered here and applying the regularisation procedure in the global bifurcation analysis, we can also deduce the following analogous result of~\cref{thm:existence}. Here \({ C \coloneqq 1 }\) for the inhomogeneous operators (the value at the origin for their symbol) and \({ C \coloneqq \norm{ \mathscr{F}^{-1}( \abs{  }^{ - \alpha } ) }_{ \textnormal{L}^1(\T) }  }\) in the homogeneous case.

\begin{theorem}
For all \({ \alpha \in (0, \infty) }\) and \({ p > 1 }\), the dispersive equation~\eqref{eq:general-evolution} admits a nontrivial periodic travelling-wave solution~\({ \varphi }\) with speed \({ c < \frac{ p }{ p - 1 } C }\), with zero mean in the homogeneous case and in case~\eqref{eq:n-sgn}. The solution is even (about \({ x \in 2 \uppi \Z }\)) and strictly increasing on~\({ (-\uppi, 0) }\), and satisfies
\begin{equation*}
\max \varphi = \varphi(0) = \mu \qquad \text{and} \qquad \varphi \in
\begin{cases}
\textnormal{C}^\alpha(\T) & \text{if } \alpha \in (0, 1); \\
\textnormal{log-Lipschitz}(\T) & \text{if } \alpha = 1; \\
\textnormal{Lipschitz}(\T) & \text{if } \alpha > 1,
\end{cases}
\end{equation*}
where again \({ \mu = (c/p)^{1/(p - 1)} }\). Moreover, the estimate
\begin{equation*}
\mu - \varphi(x) \eqsim
\begin{cases}
\abs{ x }^{ \alpha } & \text{if } \alpha \in (0, 1); \\
\abs{ x \log \abs{ x } } & \text{if } \alpha = 1; \\
\abs{ x } & \text{if }  \alpha > 1,
\end{cases}
\end{equation*}
holds uniformly around the extreme point~\({ x = 0 }\).

One has that \({ \varphi }\) is smooth around~\({ - \uppi }\) in case~\eqref{eq:n-abs}, while \({ \varphi }\) is antisymmetric about \({ - \frac{ \uppi }{ 2 } }\) in case~\eqref{eq:n-sgn}, thereby featuring inverted cusps/peakons at~\({ \uppi \Z }\).
\end{theorem}

A similar statement may be formed for extreme Lipschitz waves in a generalised version of the Degasperis--Procesi equation with the nonlinearities~\eqref{eq:n} by combining the analysis here with that of~\autocite{Arn2019b,Ork2022a}.

\subsection{Outline of the analysis}

For homogeneous dispersion one has a choice as to what class of functions \({ \abs{ \textnormal{D} }^{- \alpha} }\) should act upon in the interpretation of~\eqref{eq:evolution}. We restrict our attention to functions with zero mean (\({ \int_{ \T } u(\cdot, x) \dee x = 0 }\)), but note that other alternatives such as equivalence classes of functions that differ by a constant are possible; see for instance~\autocite{MonPelSal2020d} on homogeneous Sobolev-type spaces.

We set up~\eqref{eq:evolution} in steady variables \({ u(t, x) \coloneqq \varphi(x - ct) }\) with wave speed~\({ c > 0 }\), so that, after integration, \eqref{eq:evolution} takes the form
\begin{equation} \label{eq:steady}
\abs{ \textnormal{D} }^{- \alpha} \varphi = N(\varphi; c) + \textstyle\fint_\T n(\varphi),
\end{equation}
where we have introduced \({ N(\varphi; c) \coloneqq c \varphi - n(\varphi) }\), and the mean \({ \textstyle\fint_\T n(\varphi) \coloneqq \frac{ 1 }{ 2 \uppi } \int_{ \T } n(\varphi(y)) \dee y }\) of \({ n(\varphi) }\) is the constant of integration. One may observe that
\begin{align} \label{eq:range}
N'(\varphi) > 0 \quad \Leftrightarrow \quad n'(\varphi) < c \quad \Leftrightarrow \quad
\left\{\begin{aligned}
\varphi\phantom{|} &< \mu \textnormal{ in case~\eqref{eq:n-abs}};\\
\abs{ \varphi} &< \mu \textnormal{ in case~\eqref{eq:n-sgn},}
\end{aligned}\right.
\end{align}
in which the value
\begin{equation} \label{eq:highest-value}
\mu \coloneqq (c / p)^{1/(p-1)},
\end{equation}
being the first positive critical point for~\({ N(\varphi) }\), turns out to be the maximum of the highest wave. As the regularity analysis will reveal, the quadratic nature near \({ \varphi = \mu }\), where \({ N'' }\) is strictly negative, causes in partnership with~\({ \abs{ \textnormal{D} }^{- \alpha} }\) the singular behaviour of~\({ \varphi }\) at the crest (and through, in case~\eqref{eq:n-sgn}).

With regards to the precise \({ \textnormal{C}^\alpha }\)~regularity estimates, we consider as in~\autocite{EhrWah2019a} fine details of local regularity and first- and second-order differences of both~\({ u }\), \({ \abs{ \textnormal{D} }^{- \alpha} u }\), and~\({ n(u) }\) in connection with the Hölder seminorm. We first establish global \({ \textnormal{C}^\beta }\)~regularity for all~\({ \beta < \alpha }\), then the exact \({ \alpha }\)-Hölder estimate at~\({ 0 }\), and finally global \({ \textnormal{C}^\alpha }\)~regularity with help of an interpolation argument. A~key property in this setting is that \({ \abs{ \textnormal{D} }^{- \alpha} }\) is \({ \alpha }\)-smoothing on the scale of Hölder--Zygmund spaces, and that if \({ \abs{ \textnormal{D} }^{- \alpha} u }\) is \({ (2\alpha) }\)-Hölder continuous at a point, then~\({ u }\) is \({ \alpha }\)-Hölder continuous at that point for~\({ \alpha \in \bigl( 0, \tfrac{ 1 }{ 2 } \bigr]  }\). However, when \({ \alpha > \tfrac{ 1 }{ 2 } }\) (remember that~\autocite{EhrWah2019a} corresponds to~\({ \alpha = \tfrac{ 1 }{ 2 } }\)), \({ \abs{ \textnormal{D} }^{- \alpha} u }\) passes index~\({ 1 }\) on the Hölder--Zygmund scale, and we must partially work with derivatives as in~\autocite{EhrJohCla2019a}.

When it comes to the bifurcation analysis, we first establish small-amplitude waves by the local Crandall–Rabinowitz bifurcation theorem~\autocite[Theorems~8.3.1 and~8.4.1]{BufTol2003a}. It is interesting to note that the regularisation of~\({ n }\) lightens the computation of the local bifurcation formulas; case~\eqref{eq:n-abs} acts essentially as~\({ u^2 }\) and case~\eqref{eq:n-sgn} behaves like~\({ u^3 }\). As for the construction of the highest waves, we make use of the analytic global bifurcation theory of Buffoni and Toland~\autocite{BufTol2003a}. One obtains, after ruling out certain possibilities, a global, locally analytic curve~\({s \mapsto \varphi^\epsilon(s) }\) of smooth sinusoidal waves, along which \({ \max_{x \in \T} \varphi^\epsilon(s)(x) }\) approaches a particular value~\({ \mu^\epsilon }\) depending on the wave speed and the regularisation parameter~\({ \epsilon }\) (see~\eqref{eq:highest-value} and~\eqref{eq:highest-value-reg}). Although we are not able to establish unconditional antisymmetry in case~\eqref{eq:n-sgn}, we enforce this property along the global branch by working in a subspace. Coupled with the a~priori regularity estimates for solutions touching~\({ \mu^\epsilon }\) from below, it is then possible to extract a subsequence of \({ (\varphi^\epsilon(s))_s }\) converging to a solution~\({ \varphi^\epsilon }\) with both \({ \max \varphi^\epsilon = \mu^\epsilon}\) and the exact, global \({ \alpha }\)-Hölder continuity. Finally, we show that \({ \varphi^\epsilon }\) converges to a solution of~\eqref{eq:evolution} with the same properties as~\({ \epsilon \searrow 0 }\).

The outline of the paper is as follows. In~\cref{sec:preliminaries} we focus on properties and representations of~\({ \abs{ \textnormal{D} }^{- \alpha} }\) and \({ K_\alpha }\) on~\({ \T }\) together with the relevant function spaces. In~\cref{sec:regularity} we study a~priori properties of solutions---especially, what concerns the \({ \alpha }\)-Hölder continuity when \({ \max \varphi = \mu }\), which is the most technical part. Finally, in \cref{sec:bifurcation} we first consider the local bifurcation analysis, and then study what happens at the end of the global bifurcation curve, supported by the theory in \cref{sec:regularity}.

\enlargethispage{.5\baselineskip} 

\section{Properties of \({ \abs{ \textnormal{D} }^{- \alpha} }\) and functional-analytic setting} \label{sec:preliminaries}

\subsection{Representations of the kernel}

On the real line it is well known that the inverse Fourier transform of the symbol~\({ \abs{  }^{ - \alpha }  }\) for \({ \alpha \in (0, 1) }\) equals
\begin{equation} \label{eq:kernel-on-R}
\mathscr{F}^{-1}(\abs{  }^{ - \alpha } )(x) = \gamma_\alpha \abs{ x }^{ \alpha - 1 }
\end{equation}
in the sense of (tempered) distributions, with \({ \gamma_\alpha^{-1} \coloneqq 2 \Upgamma(\alpha) \sin \left( \tfrac{ \uppi }{ 2 }(1 - \alpha) \right)  }\) using the normalisation
\begin{equation*}
(\mathscr{F} \varphi)(\xi) = \widehat{ \varphi }(\xi) = \int_{ \R } \varphi(x) \, \textnormal{e}^{- \textnormal{i}\xi x} \dee x,
\end{equation*}
so that \({ \mathscr{F}^{-1}(\varphi)(x) = \frac{ 1 }{ 2\uppi } \mathscr{F}(\varphi)(-x) }\). Here \({ \Upgamma }\) is the gamma function, and we observe that \({ \gamma_\alpha \searrow 0 }\) as \({ \alpha \searrow 0 }\) and \({ \gamma_\alpha \nearrow \infty }\) as~\({ \alpha \nearrow 1 }\). We are interested in the action of~\({ \abs{ \textnormal{D} }^{ - \alpha} }\) in the periodic setting, and by the convolution theorem this amounts to understanding the periodic kernel
\begin{equation*}
K_\alpha(x) \coloneqq \tfrac{ 1 }{ 2 \uppi } \sum_{ k \neq 0 } \abs{ k }^{ - \alpha } \textnormal{e}^{ \textnormal{i} k x}
\end{equation*}
for which
\begin{equation*}
\widehat{ K_\alpha } = \abs{  }^{ - \alpha }  \quad \textnormal{on } \Z \setminus \Set{ 0 } \qquad \textnormal{and} \qquad \abs{ \textnormal{D} }^{- \alpha} \varphi = K_\alpha \ast \varphi \quad \textnormal{on } \T.
\end{equation*}
Here \({ \varphi }\) has zero mean so that \({ \widehat{ \varphi }(0) = 0 }\), and the normalisation is again \({ \widehat{ \varphi }(k) = \int_{ \T } \varphi(x) \, \textnormal{e}^{ - \textnormal{i} k x }  \dee x }\). A~naïve application of the Poisson summation formula yields that
\begin{equation*}
\tfrac{ 1 }{ 2 \uppi } \sum_{ k \in \Z } \abs{ k }^{ - \alpha } \textnormal{e}^{ \textnormal{i} k x} \textnormal{ \enquote{\({ = }\)} } \sum_{ k \in \Z } \gamma_\alpha \abs{ x + 2 \uppi k }^{ \alpha - 1 },
\end{equation*}
which, although it is nonsense due to divergence on both sides, nevertheless suggests that the kernel on~\({ \T }\) mimics the singularity of the kernel on~\({ \R }\). In fact, \autocite[Theorem~2.17]{SteWei1971a}~establishes the following result with help of a cut-off argument in the Poisson summation formula. We include a proof of this formula for two reasons. First, it is essential in our work. Second, the technique may be useful to define the analogues on~\({ \T }\) of other known kernels on~\({ \R }\). The latter can be used to study other nonlocal weakly dispersive equations.

\begin{proposition}[\phantom{}{\protect\autocite[Theorem~2.17]{SteWei1971a}}] \label{thm:periodic-kernel}
The periodic convolution kernel may be written as
\begin{equation*}
K_\alpha = \gamma_\alpha \abs{  }^{\alpha - 1 }  + K_{ \alpha, \textnormal{reg}}
\end{equation*}
on \({ (-\uppi, \uppi) }\), where \({ K_{ \alpha, \textnormal{reg}} }\) is an even, smooth function. In particular, \({ K_\alpha \in \textnormal{L}^1(\T)}\).
\end{proposition}
\begin{proof}
Let \({ \varrho }\) be an even, smooth cut-off function that vanishes in a neighbourhood of~\({ \xi = 0 }\) and equals~\({ 1 }\) for~\({ \abs{ \xi } \geq 1 }\), and define \({ F(\xi) \coloneqq \varrho(\xi) \abs{ \xi }^{ - \alpha }  }\) for \({ \xi \in \R }\). Then \({ F }\) is the Fourier transform of an integrable function of the form
\begin{equation*}
f(x) \coloneqq \gamma_\alpha \abs{ x }^{ \alpha - 1 } + f_{ \textnormal{reg}}(x),
\end{equation*}
where \({ f_{ \textnormal{reg}} \in \textnormal{C}^\infty(\R) }\) and \({ \abs{ f(x) } = \mathcal{O}( \abs{ x }^{ -m } )  }\) as \({ \abs{ x } \to \infty }\) for all~\({ m \geq 1 }\). Indeed, writing
\begin{equation*}
F = \abs{  }^{ - \alpha } + (\varrho - 1) \abs{  }^{ - \alpha },
\end{equation*}
it follows directly from~\eqref{eq:kernel-on-R} that \({ f = \mathscr{F}^{-1}F }\) has the given form, where we remember that the inverse Fourier transform of an integrable function of bounded support (in this case \({ (\varrho - 1) \abs{  }^{ - \alpha } }\)) is smooth. Since \({ \mathscr{F}( x^m f(x) ) \sim F^{(m)} }\) is integrable for all~\({ m \geq 1 }\), it must be the case that \({ x^m f(x) }\) is bounded. In particular, \({ f \in \textnormal{L}^1(\R) }\), and the Poisson summation formula then gives
\begin{equation*}
\sum_{ k \in \Z } f(x - 2\uppi k) = \tfrac{ 1 }{ 2\uppi } \sum_{ k \in \Z } F (k) \, \textnormal{e}^{ \textnormal{i}k x} = \tfrac{ 1 }{ 2 \uppi } \sum_{ k \neq 0 } \abs{ k }^{ - \alpha } \textnormal{e}^{ \textnormal{i}k x } = K_\alpha(x).
\end{equation*}
Since
\begin{equation*}
\sum_{ k \in \Z } f(x - 2\uppi k) = f(x) + \sum_{ k \neq 0 } f(x - 2\uppi k),
\end{equation*}
this proves the result with \({ K_{ \alpha, \textnormal{reg}} \coloneqq f_{ \textnormal{reg}} +  \sum\limits_{ k \neq 0 } f( \cdot - 2\uppi k) }\).
\end{proof}

\begin{figure}[h!]
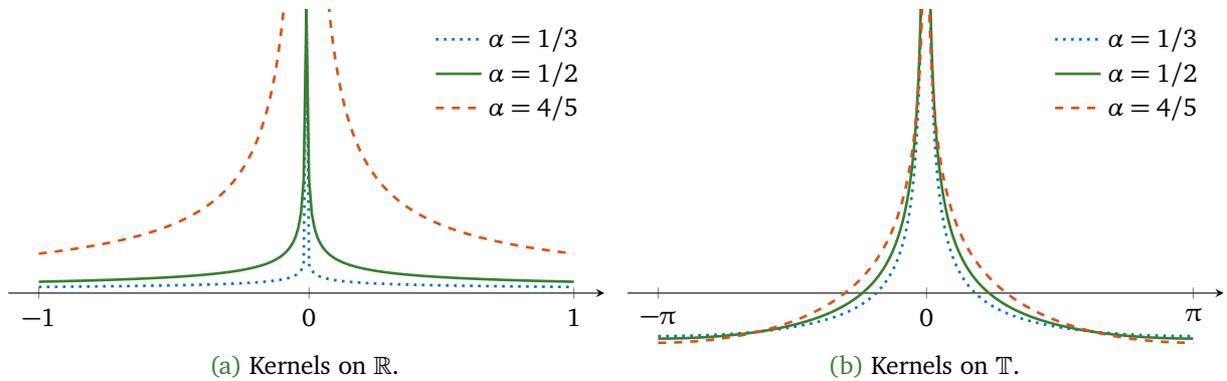

	\vspace*{-.5em}%
    \centering%
    \begin{subfigure}[b]{0.49\textwidth}%
    \centering%
\setlength\figureheight{.6\textwidth}%
\setlength\figurewidth{\textwidth}%
\input{tikz/kernels-real.tikz}%
\vspace*{-1em}%
        \caption{Kernels on~\({ \R }\).}
    \end{subfigure}
    ~ %
    \begin{subfigure}[b]{0.49\textwidth}%
\centering%
\setlength\figureheight{.6\textwidth}%
\setlength\figurewidth{\textwidth}%
\input{tikz/kernels-periodic.tikz}%
\vspace*{-1em}%
        \caption{Kernels on~\({ \T }\).}
        \label{fig:kernels-torus}
    \end{subfigure}
    \caption{Illustrating the differences between the singular kernels on~\({ \R }\) (cut at \({ x = \pm 1 }\)) and~\({ \T }\) for various~\({ \alpha }\). (The vertical axes also have different scaling.) We compute the kernels on~\({ \T }\) by numerically integrating the formula in \cref{thm:K-explicit}.}
    \label{fig:kernels}
\end{figure}
In \cref{fig:kernels} we display the integral kernels on both \({ \R }\) and \({ \T }\) for various values of~\({ \alpha }\). Whereas the kernels on~\({ \R }\) are all nonnegative, those on~\({ \T }\) become negative away from the positive singularity at~\({ 0 }\), because \({ K_\alpha }\) has zero mean. In both cases the profiles are monotone on either side of the singularity; this is obvious on~\({ \R }\), and on~\({ \T }\) we deduce this by means of the following integral representation of~\({ K_\alpha }\), which is valid for all~\({ \alpha \in (0, \infty) }\). Although the formula is known~\autocite[Section~5.4.3]{PruBryMar1986a}, we include a slick computation of it using the gamma distribution. We remark that we shall only need the monotonicity of \({ K_\alpha }\) in our work, and for that property one may alternatively use the theory of completely monotone sequences and the discrete analogue of Bernstein's theorem; see~\autocite[Theorem~3.6~b)]{BruDha2021a}.

\begin{proposition}\label{thm:K-explicit}
For all \({ \alpha \in (0, \infty) }\) the periodic kernel has the integral representation
\begin{equation*}
K_\alpha(x) = \frac{ 1 }{ \uppi \, \Upgamma( \alpha) } \int_{ 0 }^{ \infty } \frac{ t^{ \alpha - 1} ( \textnormal{e}^{ t } \cos x - 1) }{ 1 - 2 \textnormal{e}^{ t } \cos x + \textnormal{e}^{ 2t }  } \dee t
\end{equation*}
for \({ x \notin 2\uppi \Z }\). In particular, \({ K_\alpha }\) is strictly increasing on~\({ (-\uppi, 0)  }\).
\end{proposition}
\begin{proof}
By recognising \({ k^{ - \alpha} }\) in the definition of the gamma distribution with shape~\({ \alpha }\) and rate~\({ k }\), whose probability density function is \({ t \mapsto k^{ \alpha} t^{ \alpha - 1} \textnormal{e}^{ -kt }\!/ \Upgamma( \alpha) }\) on~\({ (0, \infty) }\), we find that
\pagebreak 
\begin{align*}
\uppi \, \Upgamma( \alpha) \, K_\alpha(x) &= \Upgamma( \alpha) \sum_{ k = 1 }^{ \infty } k^{- \alpha} \cos (kx) \\[1ex]
&=  \sum_{ k = 1 }^{ \infty } \int_{ 0 }^{ \mathrlap{\infty} } t^{ \alpha - 1} \textnormal{e}^{-kt} \cos (kx) \dee t \\[1ex]
&= \int_{ 0 }^{ \mathrlap{\infty} } t^{ \alpha - 1} \Re \Biggl( \sum_{ k = 1 }^{ \infty } \left(\textnormal{e}^{ -t } \textnormal{e}^{ \textnormal{i} x } \right)^k \Biggr) \dee t \\[1ex]
&= \int_{ 0 }^{ \mathrlap{\infty} } t^{ \alpha - 1} \Re \biggl( \frac{ \textnormal{e}^{ \textnormal{i}x }  }{ \textnormal{e}^{ t } - \textnormal{e}^{ \textnormal{i}x }  } \biggr) \dee t = \int_{ 0 }^{ \infty } \frac{ t^{ \alpha - 1} ( \textnormal{e}^{ t } \cos x - 1) }{ 1 - 2 \textnormal{e}^{ t } \cos x + \textnormal{e}^{ 2t }  } \dee t
\end{align*}
using the dominated convergence theorem and a trick with geometric series. Leibniz's integral rule next yields that
\begin{equation*}
K_\alpha'(x) = -\frac{ \sin x }{ \uppi \, \Upgamma( \alpha) } \int_{ 0 }^{ \mathrlap{\infty} } \frac{ t^{ \alpha - 1} \textnormal{e}^{ t } ( \textnormal{e}^{ 2t }  - 1) }{ (1 - 2 \textnormal{e}^{ t } \cos x + \textnormal{e}^{ 2t })^2 } \dee t,
\end{equation*}
which shows that \({ K_\alpha }\) is strictly increasing on~\({ (-\uppi, 0)  }\).
\end{proof}

\vspace*{-.5em} 

\begin{remark}
For \({ \alpha = 1 }\) in \cref{thm:K-explicit} one finds the explicit form
\begin{align*}
K_1(x) &= \tfrac{ 1 }{ \uppi } \Bigl[\tfrac{ 1 }{ 2 } \log \left( 1 - 2 \textnormal{e}^{ t } \cos x + \textnormal{e}^{ 2t } \right) - t \Bigr]_{t = 0}^\infty \\[1ex]
&= - \tfrac{ 1 }{ 2 \uppi } \log \bigl( 2(1 - \cos x) \bigr) \\[1ex]
&= - \tfrac{ 1 }{ \uppi } \log \abs{ x } + \mathcal{O}(x^2)
\end{align*}
as~\({ x \to 0 }\). Such logarithmic singularities occur for all the kernels in \cref{tab:overview} when \({ \alpha = 1 }\); see for instance~\autocite[Lemma~2.3~(iii)]{EhrJohCla2019a} for the kernel in the bidirectional Whitham equation with inhomogeneous dispersion~\({ \tanh(\textnormal{D})/ \textnormal{D} }\).
\end{remark}

\vspace*{-1.5em} 

\subsection{Action of \({ \abs{ \textnormal{D} }^{- \alpha} }\) on Hölder--Zygmund spaces}

As regards the functional-analytic framework, it is desirable to work with spaces which both capture the precise regularity of cusps and interact well with Fourier multipliers. These turn out to be the so-called Hölder--Zygmund spaces, which we explain next.

Let \({ \textnormal{C}^m(\T) }\), for \({ m = 0, 1, \dotsc }\), denote the space of \({ m }\)~times continuously differentiable functions~\({ \varphi }\) on~\({ \T }\) with norm
\begin{equation*}
\norm{ \varphi }_{ \textnormal{C}^{m}(\T) } \coloneqq \norm{ \varphi }_{ \infty } + \norm{ \varphi^{(m)} }_{ \infty },
\end{equation*}
where \({ \norm{  }_{ \infty }  }\) is the supremum norm on~\({ \T }\). Furthermore, define \({ \textnormal{C}^{m, \beta}(\T) }\) with \({ \beta \in (0, 1] }\) to be the class of Hölder spaces consisting of all \({ \varphi \in \textnormal{C}^m(\T) }\) for which \({ \varphi^{(m)} }\) is \({ \beta }\)-Hölder continuous with norm
\begin{equation*}
\norm{ \varphi }_{ \textnormal{C}^{m, \beta}(\T) } \coloneqq \norm{ \varphi }_{ \textnormal{C}^{m}(\T) } + \seminorm{ \varphi^{(m)} }_{ \textnormal{C}^{ \beta}(\T) },
\end{equation*}
where
\begin{equation*}
\seminorm{ \psi }_{ \textnormal{C}^\beta(\T) } \coloneqq \sup_{ x \neq y } \frac{ \abs{ \psi(x) - \psi(y) } }{ \abs{ x - y }^{ \beta }  }
\end{equation*}
is a seminorm. We write \({ \textnormal{C}(\T) \coloneqq \textnormal{C}^0(\T) }\) and \({ \textnormal{C}^\beta(\T) \coloneqq \textnormal{C}^{0, \beta}(\T) }\) for simplicity, and note that \({ \textnormal{C}^{m, \beta}(\T) }\) is compactly embedded in \({ \textnormal{C}^{m, \widetilde{\beta}}(\T) }\) when \({ \beta > \widetilde{\beta}}\). Moreover, the Fourier series of \({ \varphi \in \textnormal{C}^{ \beta }(\T) }\) for \({ \beta > \frac{ 1 }{ 2 } }\) converges both uniformly to~\({ \varphi }\) and absolutely.

While the standard Hölder norms provide an accurate description of the modulus of continuity of a function (and its derivatives), an alternative, frequency-based characterisation by means of the Littlewood--Paley decomposition is more suitable for Fourier multipliers. To this end, let \({  \sum_{ j = 0 }^\infty \varrho_j(\xi) = 1 }\) be a partition of unity of smooth functions \({ \varrho_j }\) on~\({ \R }\) supported on \({ 2^j \leq \abs{ \xi } \leq 2^{j + 1} }\) for \({ j \geq 1 }\) and on \({ \abs{ \xi } \leq 2 }\) for~\({ j = 0 }\). We then define the Hölder--Zygmund space~\({ \textnormal{C}_\ast^s(\T) }\) for \({ s \in [0, \infty) }\) to consist of those functions~\({ \varphi }\) for which
\begin{equation*}
\norm{ \varphi }_{ \textnormal{C}_\ast^s(\T) } \coloneqq \sup_{ j \geq 0 } 2^{js} \norm{ \varrho_j( \textnormal{D}) \varphi  }_{ \infty }
\end{equation*}
is finite, where \({ \varrho_j(\textnormal{D}) }\) is the Fourier multiplier with symbol~\({ \varrho_j }\), that is,
\begin{equation*}
\smash[b]{\varrho_j(\textnormal{D}) \varphi(x) \coloneqq \sum_{ k \in \Z } \varrho_j(k) \, \widehat{ \varphi }(k) \, \textnormal{e}^{ \textnormal{i}k x}.}
\end{equation*}
One has that\vspace*{-.3em} 
\begin{align*} \SwapAboveDisplaySkip
\textnormal{C}_\ast^s(\T) = \textnormal{C}^{ \lfloor s \rfloor, s - \lfloor s \rfloor}(\T) \quad &\textnormal{for } s \neq 1, 2, \dotsc,
\intertext{in the sense of equivalent norms, whereas there are strict inclusions}
\textnormal{C}^{s}(\T) \subsetneq \textnormal{C}^{s -1, 1}(\T) \subsetneq \textnormal{C}_\ast^s(\T) \quad &\textnormal{for } s = 1, 2, \dotsc
\end{align*}

Since \({ \abs{ \textnormal{D} }^{- \alpha} }\) is homogeneous, we restrict from now on to the corresponding subspaces \({ \mathring{ \textnormal{C}}^{m}(\T) }\), \({ \mathring{ \textnormal{C}}^{m, \beta}(\T) }\), and \({ \mathring{ \textnormal{C}}_\ast^{s}(\T) }\) of functions with zero mean in the above spaces, with identical norms. Note that the seminorm \({ \seminorm{}_{ \mathring{ \textnormal{C} }^\beta(\T)} }\) is now a norm equivalent to \({ \norm{  }_{  \mathring{ \textnormal{C}}^{0, \beta}(\T) }  }\) by the zero-mean restriction and compactness of~\({ \T }\). We observe in this setting that
\begin{equation*}
\abs{ \textnormal{D} }^{- \alpha} \colon \mathring{\textnormal{C}}_\ast^s(\T) \to \mathring{\textnormal{C}}_\ast^{s + \alpha}(\T)
\end{equation*}
is \({ \alpha }\)-smoothing for all \({ s \in [0, \infty) }\) (and therefore also on the Hölder-space scale for \({ s \neq 1, 2, \dotsc }\)), because in terms of Fourier multipliers we have, with \({ j \geq 1 }\), that
\begin{equation*}
\mathscr{F} \left( \varrho_j(\textnormal{D}) \abs{ \textnormal{D} }^{- \alpha} \varphi \right) = \varrho_j \abs{  }^{ - \alpha } \widehat{ \varphi } \sim 2^{-j \alpha} \varrho_j \widehat{ \varphi } = 2^{-j \alpha} \widehat{ \varrho_j(\textnormal{D}) \varphi }.
\end{equation*}

Finally, the subscript \enquote{even} attached to any space on~\({ \T }\) means the subspace of symmetric functions about~\({ 0 \pmod{2 \uppi} }\).

\begin{lemma} \label{thm:local-smoothing}
The smoothing property \({ \abs{ \textnormal{D} }^{- \alpha} \colon \mathring{\textnormal{C}}_\ast^s(\T) \to \mathring{\textnormal{C}}_\ast^{s + \alpha}(\T) }\) extends to a local version. More precisely, if \({ \varphi \in \mathring{ \textnormal{C}}(\T) }\) lies in \({ \textnormal{C}_{\ast, \textnormal{loc}}^{s}(U) }\) for an open subset~\({ U \subset \T }\), in the sense that \({ \rho \varphi \in \textnormal{C}_\ast^s(\T) }\) for any compactly supported \({ \rho \in \textnormal{C}_{ \textnormal{c}}^\infty(U)  }\), then we still have \({ \abs{ \textnormal{D} }^{- \alpha} \varphi \in \textnormal{C}_{\ast, \textnormal{loc}}^{s + \alpha}(U) }\).
\end{lemma}
\vspace*{-.8em} 
\enlargethispage{1.2\baselineskip} 
\begin{proof}
To see this, let \({ \rho \in \textnormal{C}_{ \textnormal{c}}^\infty(U) }\) and let \({ \eta \in \textnormal{C}_{ \textnormal{c}}^\infty(U) }\) satisfy \({ \eta = 1 }\) in a neighbourhood \({ V \Subset U }\) of~\({ \support \rho }\). Then
\begin{equation*}
\rho \abs{ \textnormal{D} }^{- \alpha} \varphi = \rho \abs{ \textnormal{D} }^{- \alpha} (\eta \varphi) + \rho \abs{ \textnormal{D} }^{- \alpha} \bigl( (1 - \eta) \varphi \bigr),
\end{equation*}
and the first term on the right-hand side is globally \({ \textnormal{C}_\ast^{s + \alpha} }\) regular. Moreover, since the integrand in
\begin{equation*}
\rho(x) \abs{ \textnormal{D} }^{- \alpha} \bigl( (1 - \eta) \varphi \bigr)(x) = \int_{ -\uppi }^{ \uppi } K_\alpha(x - y) \rho(x) (1 - \eta(y)) \varphi(y) \dee y
\end{equation*}
vanishes for \({ y }\) near~\({ x }\) and \({ K_\alpha }\) is smooth away from~\({ 0 }\), it follows that \({ \rho  \abs{ \textnormal{D} }^{- \alpha} \bigl( (1 - \eta) \varphi \bigr) }\) is smooth. Hence, \({ \abs{ \textnormal{D} }^{- \alpha} \varphi \in \textnormal{C}_{\ast, \textnormal{loc}}^{s + \alpha}(U)  }\), as claimed.
\end{proof}

Finally, we include a monotonicity property of~\({ \abs{ \textnormal{D} }^{- \alpha} }\) which will be useful in establishing a priori nodal properties of highest waves in \cref{sec:regularity}.
\begin{proposition}[\phantom{}{\protect\autocite[Lemma~3.6]{EhrWah2019a}}] \label{thm:L-paritypreserving}
\({ \abs{ \textnormal{D} }^{- \alpha} }\) is a parity-preserving operator, and \({ \abs{ \textnormal{D} }^{- \alpha}f > 0 }\) on \({ (-\uppi, 0) }\) for odd \({ f \in \mathring{ \textnormal{C}}(\T) }\) satisfying \({ f \geq 0 }\) on~\({ (-\uppi, 0) }\) with \({ f(y_0) > 0 }\) for some~\({ y_0 \in (-\uppi, 0) }\).
\end{proposition}
\begin{proof}
Since \({ K_\alpha }\) is even, one immediately obtains that \({ \abs{ \textnormal{D} }^{- \alpha} }\) is parity-preserving from
\begin{equation*}
\abs{ \textnormal{D} }^{- \alpha}f(x) \pm \abs{ \textnormal{D} }^{- \alpha}f(-x) = \int_{ \T } K_\alpha(x - y) \left( f(y) \pm f(-y) \right) \dee y.
\end{equation*}
Next consider odd \({ f \in \mathring{ \textnormal{C}}(\T) }\) satisfying \({ f(y) \geq 0 }\) on~\({ (- \uppi, 0) }\) with \({ f(y_0) > 0 }\) for some~\({ y_0 \in (-\uppi, 0) }\). Then
\begin{equation*}
\abs{ \textnormal{D} }^{- \alpha}f(x) = \int_{ -\uppi }^{ \uppi } K_\alpha(x - y) f(y) \dee y = \int_{ - \uppi }^{ 0 } \left( K_\alpha(x - y) - K_\alpha(x + y) \right) f(y) \dee y
\end{equation*}
for \({ x \in (- \uppi, 0) }\). When \({ y }\) also lies in \({ (- \uppi, 0) }\), it follows that \({ - 2 \uppi < x + y < x -y < \uppi }\) and
\begin{equation*}
\distance{x - y}{0} < \min \Set[\big]{ \distance{x + y}{0}, \distance{x + y}{-2 \uppi} }.
\end{equation*}
The latter inequality is a consequence of \({ \abs{ x - y } < \abs{ x + y } }\) for \({ x,y < 0 }\) and \({ \abs{ x - y } < x + y + 2 \uppi }\) for \({ x, y > - \uppi}\). Since \({ K_\alpha }\) is strictly decreasing as a function of the distance from the origin to~\({ \pm \uppi }\) by \cref{thm:K-explicit} and is even and periodic, we therefore obtain that
\begin{equation*}
K_\alpha(x - y) > K_\alpha(x + y)
\end{equation*}
for every \({ y \in (- \uppi, 0) \setminus \Set{ x } }\). In particular, \({ \abs{ \textnormal{D} }^{- \alpha}f(x) > 0 }\) for all \({ x \in (- \uppi, 0) }\) as \({ f }\) is strictly positive in an interval around~\({ y_0 }\) by continuity.
\end{proof}

\section{A priori properties of travelling-wave solutions} \label{sec:regularity}

In this section we establish many a priori bounds and regularity properties of continuous solutions of~\eqref{eq:steady}. Most importantly, we prove exact, global \({ \alpha }\)-Hölder regularity in \cref{thm:regularity} for solutions touching the highest point~\({ \mu }\) (see~\eqref{eq:highest-value}) from below at~\({ x = 0 }\). This is obtained with help of a nodal pattern of solutions in \cref{thm:nodal-pattern}. We remind the reader of~\eqref{eq:range}: \({ n'(\varphi) < c }\) corresponds to solutions which stay away from~\({ (\pm)\mu }\), and \({ n'(\varphi) \leq c }\) includes the possibility of also touching~\({ (\pm) \mu }\).

Our first result is a uniform upper bound on both the size of solutions and the wave speed that we will use in \cref{sec:bifurcation} in compactness arguments.
\begin{lemma} \label{thm:bound-wavespeed}
For all solutions \({ \varphi \in \mathring{\textnormal{C}}(\T) }\) of~\eqref{eq:steady} satisfying \({ n'(\varphi) \leq c }\) one has the
uniform estimate
\begin{equation*}
\norm{ \varphi }_{ \infty } \lesssim (1 + c)^{1/(p - 1)}.
\end{equation*}
Moreover, if \({ c \geq \frac{ p }{ p - 1 } \norm{ K_\alpha }_{ \textnormal{L}^1(\T) } }\), then there are no nontrivial such solutions.
\end{lemma}
\vspace*{-.5em} 
\begin{proof}
In case~\eqref{eq:n-sgn} for \({ n'(\varphi) \leq c }\), the bound in \({ \textnormal{L}^\infty }\) is immediate since~\({ \mu \sim c^{1/(p - 1)} }\). As regards~\eqref{eq:n-abs}, we need to control the minimum of a nontrivial~\({ \varphi }\). Let \({ x_\textnormal{min} }\) and \({ x_\textnormal{max} }\) be points where \({ \varphi }\) attains its global minimum~\({ \varphi_\textnormal{min} }\) and maximum~\({ \varphi_\textnormal{max} }\), respectively, where we note that \({ \varphi_\textnormal{max} > 0 > \varphi_\textnormal{min} }\) as \({ \varphi }\) has zero mean. From~\eqref{eq:steady} we then find that
\pagebreak 
\begin{align} \label{eq:max-min-estimate}
\begin{aligned}
c \,\bigl( \varphi_\textnormal{max} - \varphi_\textnormal{min} \bigr) - \bigl( n(\varphi_\textnormal{max}) - n(\varphi_\textnormal{min}) \bigr) &= \abs{ \textnormal{D} }^{- \alpha} \varphi(x_\textnormal{max}) - \abs{ \textnormal{D} }^{- \alpha} \varphi(x_\textnormal{min}) \\[1ex]
&\leq \norm{ K_\alpha }_{ \textnormal{L}^1(\T) } \bigl( \varphi_\textnormal{max} - \varphi_\textnormal{min} \bigr),
\end{aligned}
\end{align}
which leads to
\begin{align*} \SwapAboveDisplaySkip
n(\varphi_\textnormal{min}) &\leq n(\varphi_\textnormal{max}) + ( \norm{ K_\alpha }_{ \textnormal{L}^1(\T) } - c ) ( \varphi_\textnormal{max} + \abs{\varphi_\textnormal{min}} ) \\[1ex]
& \lesssim \max \Set[\big]{ n(\varphi_\textnormal{max}), (1 + c)  ( \varphi_\textnormal{max} + \abs{\varphi_\textnormal{min}} )}.
\end{align*}
In the first situation, \({ n(\varphi_\textnormal{min}) \lesssim n(\varphi_\textnormal{max}) }\), so that \({ \abs{\varphi_\textnormal{min}} \lesssim \varphi_\textnormal{max} \leq \mu }\). In the second situation, it suffices to investigate the case \({ \varphi_\textnormal{max} < \abs{ \varphi_\textnormal{min} } }\) (otherwise we freely get \({ \abs{\varphi_\textnormal{min}} \leq  \varphi_\textnormal{max} \leq \mu }\)). Then
\begin{equation*}
\abs{ \varphi_\textnormal{min} }^{p} = n(\varphi_\textnormal{min}) \lesssim (1 + c) \abs{ \varphi_\textnormal{min} },
\end{equation*}
implying that \({ \abs{ \varphi_\textnormal{min} } \lesssim (1 + c)^{1/(p - 1)} }\).

For the last part, we use in case~\eqref{eq:n-abs} that for \({ a > 0 > b }\) one has
\begin{equation*}
a^{p} - \abs{ b }^{p} = a^{p - 1} (a - b) - \abs{ b }( a^{p - 1} + \abs{ b }^{p - 1}) < a^{p - 1} (a - b).
\end{equation*}
Applied to \({ a = \varphi_\textnormal{max} }\) and \({ b = \varphi_\textnormal{min} }\), we reorder~\eqref{eq:max-min-estimate} and find that
\begin{equation*}
c - \norm{ K_\alpha }_{ \textnormal{L}^1(\T) } \leq \frac{ n(\varphi_\textnormal{max}) - n(\varphi_\textnormal{min}) }{ \varphi_\textnormal{max} - \varphi_\textnormal{min} } < \varphi_\textnormal{max}^{p - 1} \leq \mu^{p - 1} = \frac{ c }{ p },
\end{equation*}
which gives \({ c < \frac{ p }{ p - 1 } \norm{ K_\alpha }_{ \textnormal{L}^1(\T) } }\). One may similarly treat case~\eqref{eq:n-sgn}.
\end{proof}

Next, we want to establish regularity for solutions which stay away from~\({ (\pm) \mu }\). To this end, we need the inverse function theorem for the Hölder scale. We could not find a proof in the literature and therefore provide a short argument.

\begin{proposition}[Inverse function theorem for \({ \textnormal{C}^{m, \beta} }\)] \label{thm:inverse-function}
Let \({ m \geq 1 }\) be an integer and \({ \beta \in (0, 1] }\), and assume that a function~\({ f }\) is \({ \textnormal{C}^{m, \beta} }\)~regular and strictly monotone on a compact interval~\({ I \subset \R }\). Then \({ f^{-1} }\) is \({ \textnormal{C}^{m, \beta} }\)~regular on~\({ f(I) }\).
\end{proposition}
\begin{proof}
We only establish \({ \beta }\)-Hölder regularity of \({ g' }\), where~\({ g \coloneqq f^{-1} }\); the rest follows by the standard inverse function theorem and similar estimates for the higher-order derivatives. To this end, let \({ x, y \in f(I) }\) be different and observe that
\begin{align*}
\frac{ \abs{ g'(x) - g'(y) } }{ \abs{ x - y }^{ \beta }  } &= \frac{ \abs[\Big]{ \displaystyle \frac{ 1 }{ f'(g(x)) } - \frac{ 1 }{ f'(g(y)) } } }{ \abs{ x - y }^{ \beta }  } \\[1ex]
&= \frac{ 1 }{ \abs{ f'(g(x)) \, f'(g(y)) } } \cdot \frac{ \abs{ f'(g(x)) - f'(g(y)) } }{ \abs{ g(x) - g(y) }^\beta } \cdot \abs*{ \frac{ g(x) - g(y) }{ x - y } }^\beta \\[1ex]
&= \abs{ g'(x) \, g'(y) } \cdot \frac{ \abs{ f'(g(x)) - f'(g(y)) } }{ \abs{ g(x) - g(y) }^\beta } \cdot \abs{ g'(z) }^\beta
\end{align*}
for some \({ z }\) between \({ x }\) and \({ y }\) by the mean value theorem. Hence, \({ \seminorm{g'}_{ \textnormal{C}^{\beta}} \leq \norm{ g' }_{ \infty }^{ 2 + \beta } \seminorm{ f' }_{ \textnormal{C}^\beta } < \infty }\), where we note that \({ g' }\) is bounded on the compact set~\({ f(I) }\) due to its continuity by the standard inverse function theorem.
\end{proof}

\begin{lemma} \label{thm:smoothness}
Let \({ \varphi \in \mathring{\textnormal{C}}(\T) }\) be a solution of~\eqref{eq:steady}. Then
\begin{enumerate}[ref=\roman*)]
\item \label{thm:smoothness-open}
\({ \varphi }\) is smooth on any open set where \({ n'(\varphi) < c }\) and that does not contain the boundary~\({ \partial(\varphi^{-1}(0)) }\) of the set~\({ \varphi^{-1}(0) }\); and
\item \label{thm:smoothness-zero}
\({ \varphi }\) has at least the same regularity in the Hölder scale around \({ \partial( \varphi^{-1}(0)) }\) as the nonlinearity~\({ n }\) around~\({ 0 }\).
\end{enumerate}
In particular, if \({ n }\) is smooth, then so is \({ \varphi }\) on any open set where~\({ n'(\varphi) < c }\).
\end{lemma}
\begin{proof}
Note first by translation invariance (\({ \abs{ \textnormal{D} }^{- \alpha} }\) is a convolution operator) that if \({ \varphi }\) is a solution of~\eqref{eq:steady}, then so is \({ \varphi(\cdot + h) }\) for any~\({ h \in \R }\). Accordingly, it suffices to consider open sets~\({ U \subseteq (-\uppi, \uppi) }\) where \({ n'(\varphi) < c }\), so that~\eqref{eq:range} holds uniformly on every compact interval~\({ I \subset U }\). By the inverse function theorem (\cref{thm:inverse-function}), it follows that \({ N^{-1} }\) exists on \({ \varphi(I) }\) and has the same regularity as \({ n }\) in the Hölder scale. As such, we may then invert~\eqref{eq:steady} to get the pointwise relation
\begin{equation} \label{eq:composition-operator}
\varphi(x) = G( \varphi, c )(x) \coloneqq N^{-1} \bigl( \abs{ \textnormal{D} }^{- \alpha} \varphi(x) - \textstyle\fint_\T n(\varphi) \bigr)
\end{equation}
for \({ x \in I }\), where \({ G }\) is a nonlinear composition operator, depending implicitly on~\({ c }\) via~\({ N^{-1} }\).

It is clear from the (higher-order) chain rule that an operator \({ f \mapsto F \circ f }\) maps the space \({ \textnormal{C}^m(I) }\) into itself provided that \({ F \in \textnormal{C}^m_{ \textnormal{loc}}(\R) }\). More generally, the same remains true for \({ \textnormal{C}^{m, \beta}(I) }\) for all \({ m \in \N_0 }\) and \({ \beta \in (0, 1] }\) if \({ F \in \textnormal{C}^{m, \beta}_{ \textnormal{loc}}(\R) \cap \textnormal{C}^{0, 1}_{ \textnormal{loc}}(\R) }\) by~\autocite[Theorems~2.1, 4.1 and~5.1]{GoeSac1999a}, and the composition operator is also bounded (maps bounded sets to bounded sets). Therefore, since \({ \varphi \in \mathring{ \textnormal{C}}(\T) \hookrightarrow \textnormal{C}_{\ast, \textnormal{loc}}^{0}(U) }\) and \({ \abs{ \textnormal{D} }^{- \alpha} }\) is locally \({ \textnormal{C}_\ast^\alpha }\)~smoothing by \cref{thm:local-smoothing}, it follows by bootstrapping of \({ G(\cdot, c) }\) in~\eqref{eq:composition-operator} that \({ \varphi }\) has at least the same \({ \textnormal{C}^{m, \beta} }\) Hölder regularity as \({ N^{-1} }\) (that is, as~\({ n }\)) on~\({ I }\). In particular, \({ \varphi }\) has the given regularity around \({ \partial (\varphi^{-1}(0)) }\) by applying this result to a covering \({ \Set{ I_j }_j }\) of compact intervals~\({ I_j }\) such that \({ (- \uppi, \uppi) \supset \bigcup_j I_j^{} \supset \partial(\varphi^{-1}(0))}\), which proves property~\ref{thm:smoothness-zero}.

Similarly, when \({ U }\) does not contain \({ \partial (\varphi^{-1}(0)) }\), we know that \({ N^{-1} }\) is smooth on \({ \varphi(I) \not\ni 0 }\), and so bootstrapping~\eqref{eq:composition-operator} yields that \({ \varphi }\) is smooth on~\({ I }\). As \({ I \subset U }\) was arbitrary, this establishes property~\ref{thm:smoothness-open}.
\end{proof}

We continue by proving a nodal pattern for solutions which stay away from~\({ \pm \mu }\). The result will be crucial in establishing that the global bifurcation branch of solutions in \cref{sec:bifurcation} is not periodic.

\begin{theorem}[Nodal pattern] \label{thm:nodal-pattern}
Let \({ \varphi \in \mathring{\textnormal{C}}_{\textnormal{even}}(\T) }\) be a nontrivial solution of~\eqref{eq:steady} that is increasing on~\({ (- \uppi, 0) }\). If \({ \varphi \in \mathring{\textnormal{C}}_{\textnormal{even}}^1(\T) }\), then
\begin{equation*}
\varphi' > 0 \quad \text{and} \quad n'(\varphi) < c \quad \text{on} \quad (- \uppi, 0),
\end{equation*}
and \({ \varphi }\) has the regularity specified in \cref{thm:smoothness}. Moreover, if \({ \varphi }\) is also \({ \textnormal{C}^2 }\) regular around~\({ 0 }\) and \({ - \uppi }\) (in the sense of~\({ \T }\)), then \({ n'(\varphi) < c }\) everywhere,
\begin{equation*}
\varphi''(0) < 0, \quad \text{and} \quad \varphi''(- \uppi) > 0.
\end{equation*}

Conversely, if \({ n'(\varphi) \leq c }\) everywhere, then \({ \varphi }\) features the regularity in \cref{thm:smoothness}, with
\begin{equation*}
\varphi' > 0 \quad \text{on} \quad (- \uppi, 0).
\end{equation*}
\end{theorem}

\begin{remark}
We write \enquote{in-/decreasing} instead of \enquote{nonde-/increasing}, so that constant functions are both increasing and decreasing, and add the prefix \enquote{strictly} for the nontrivial cases.
\end{remark}

\begin{proof}
In the first case, \({ \varphi' }\) is odd and satisfies \({ \varphi' \geq 0 }\) on~\({ (-\uppi, 0) }\) with \({ \varphi'(y_0) > 0 }\) for some~\({ y_0 \in (- \uppi, 0) }\), as \({ \varphi }\) is nonconstant and increasing, so that \({ \abs{ \textnormal{D} }^{- \alpha} \varphi' > 0 }\) on \({ (- \uppi, 0) }\) by \cref{thm:L-paritypreserving}. We then differentiate in~\eqref{eq:steady} to find that
\begin{equation*}
N'(\varphi) \varphi' = \abs{ \textnormal{D} }^{- \alpha} \varphi' > 0 \quad \text{on} \quad (- \uppi, 0),
\end{equation*}
which implies that both \({ N'(\varphi) = c - n'(\varphi) }\) and \({ \varphi' }\) are strictly positive on that interval.

If \({ \varphi }\) is also \({ \textnormal{C}^2 }\) around~\({ 0 }\), we differentiate~\eqref{eq:steady} twice and use that \({ \varphi'(0) = 0 }\) to obtain
\begin{align*}
N'(\varphi(0)) \, \varphi''(0) &= (\abs{ \textnormal{D} }^{- \alpha} \varphi')'(0) \\[1ex]
&= \lim_{ r \searrow 0 } \frac{\dee  }{\dee x} \Biggg( \smashoperator{\int_{ \hspace*{2em}\abs{ y } < r }} K_\alpha(y) \varphi'(x - y) \dee y + \smashoperator{\int_{\uppi \geq \abs{ y } \geq r }} K_\alpha(y) \varphi'(x - y) \dee y \Biggg) \Biggr\vert_{x = 0},
\end{align*}
where we have also isolated the singularity of~\({ K_\alpha }\) and interchanged limits (which is legitimate since \({ (\abs{ \textnormal{D} }^{- \alpha} \varphi')' }\) is continuous around~\({ 0 }\)). Leibniz's integral rule now gives
\begin{equation*}
\frac{\dee  }{\dee x} \smashoperator{\int_{\hspace*{2em}\abs{ y } < r }} K_\alpha(y) \varphi'(x - y) \dee y \bigg\rvert_{x = 0} = \smashoperator{\int_{\hspace*{2em}\abs{ y } < r }} K_\alpha(y) \varphi''(y) \dee y,
\end{equation*}
and the latter integral vanishes as \({ r \searrow 0 }\) because \({ K_\alpha }\) is integrable and \({ \varphi'' }\) is continuous around~\({ 0 }\). By Leibniz's rule once more, we also find that
\begin{align*}
\tfrac{ 1 }{ 2 } \frac{\dee }{\dee x} \smashoperator{\int_{\hspace*{3em} \uppi \geq \abs{ y } \geq r }} K_\alpha(y) \varphi'(x - y) \dee y \bigg\rvert_{x = 0} &= - \tfrac{ 1 }{ 2 } \frac{\dee  }{\dee x} \smashoperator{\int_{\hspace*{4em} \uppi \geq \abs{ x - y } \geq r }} K_\alpha(x - y) \varphi'(y) \dee y \bigg\rvert_{x = 0} \\[1ex]
&= K_\alpha(r) \varphi'(r) - K_\alpha( \uppi) \underbrace{\varphi'(\uppi)}_{ = 0} - \int_{ r }^{ \uppi } K_\alpha'(y) \varphi'(y) \dee y,
\end{align*}
where we have utilised that \({ K_\alpha }\) is even and \({ \varphi' }\) is odd. Observe that \({ K_\alpha(r) \eqsim \abs{ r }^{ \alpha - 1} }\) and \({ \varphi'(r) = \mathcal{O}(r) }\) (because \({ \varphi'' }\) is continuous around~\({ 0 }\)) as \({ r \searrow 0 }\), which means that \({ K_\alpha(r) \varphi'(r) }\) vanishes in the limit. Since \({ K_\alpha' }\) and \({ \varphi' }\) are strictly negative on \({ (0, \uppi) }\) by \cref{thm:K-explicit} and the assumption, respectively, we further infer that \({ -\int_{ r }^{ \uppi } K_\alpha'(y) \varphi'(y) \dee y }\) is both negative and strictly decreasing as~\({ r \searrow 0 }\). As such, we obtain
\begin{equation*}
\underbrace{\left(c - n'(\varphi(0))\right)}_{= N'(\varphi(0))} \varphi''(0) = (\abs{ \textnormal{D} }^{- \alpha} \varphi')'(0) = - 2 \lim_{ r \searrow 0 } \int_{ r }^{ \uppi } K_\alpha'(y) \varphi'(y) \dee y < 0.
\end{equation*}
Since \({ n'(\varphi) < c }\) on \({ (-\uppi, \uppi) \setminus \Set{ 0 } }\) and \({ n'(\varphi) }\) is continuous, it follows that \({ n'(\varphi(0)) < c }\) also, and consequently \({ \varphi''(0) < 0 }\). By similar calculations one finds that \({ n'(\varphi(- \uppi)) < c }\) (for free in case~\eqref{eq:n-abs}) and~\({ \varphi''(- \uppi) > 0 }\).

Conversely, suppose that \({ n'(\varphi) \leq c }\) everywhere. If in fact \({ n'(\varphi) < c }\) uniformly, then \cref{thm:smoothness} implies that \({ \varphi \in \mathring{\textnormal{C}}^1(\T) }\), which leads to \({ \varphi' > 0 }\) on \({ (-\uppi, 0) }\) by the first case of \cref{thm:nodal-pattern}. When \({ n'(\varphi) }\) touches~\({ c }\), however, we must use a different approach. Note that \({ \varphi }\) is differentiable almost everywhere on \({ (- \uppi, 0) }\) by Lebesgue's theorem for increasing functions, and that we may also use central differences to compute~\({ \varphi' }\). To this end, observe that 
\begin{equation} \label{eq:double-symmetrisation}
\abs{ \textnormal{D} }^{- \alpha} \varphi(x + h) - \abs{ \textnormal{D} }^{- \alpha} \varphi(x - h) = \int_{ - \uppi }^{ 0 } \bigl( K_\alpha(y - x) - K_\alpha(y + x) \bigr) \, \bigl( \varphi(y + h) - \varphi(y - h) \bigr) \dee y
\end{equation}
for \({ x \in (- \uppi, 0) }\) and \({ h \in (0, \uppi) }\) by periodicity and evenness of \({ K_\alpha }\) and~\({ \varphi }\). The second factor in the integrand is nonnegative by assumption, whereas the first factor is strictly positive by \cref{thm:K-explicit}. Consequently, since \({ \varphi }\) is nontrivial, \({ \abs{ \textnormal{D} }^{- \alpha} \varphi }\) and therefore also \({ N(\varphi) }\) are strictly increasing on~\({ (- \uppi, 0) }\). Then for all \({ -\uppi < y < x < 0 }\) we find that
\begin{equation} \label{eq:N-increasing}
0 < N(\varphi(x)) - N(\varphi(y)) = (\varphi(x) - \varphi(y)) \underbrace{N'(\varphi(\xi))}_{ > 0}
\end{equation}
for some \({ \xi \in (y, x) }\) by the mean value and intermediate value theorems, which yields that \({ \varphi }\) is strictly increasing on \({ (- \uppi, 0) }\). Moreover, \eqref{eq:double-symmetrisation} and~\eqref{eq:N-increasing} together show that
\begin{align} \label{eq:fatou}
\begin{aligned}
N'(\varphi(x)) \, \varphi'(x) &= \lim_{ h \searrow 0 } \frac{ N(\varphi(x + h)) - N(\varphi(x - h)) }{ 2h } \\[1ex]
&= \lim_{ h \searrow 0 } \frac{ \abs{ \textnormal{D} }^{- \alpha} \varphi(x + h) - \abs{ \textnormal{D} }^{- \alpha} \varphi(x - h)  }{2h} \\[1ex]
&\geq \int_{ - \uppi }^{ 0 } \bigl( K_\alpha(y - x) - K_\alpha(y + x) \bigr) \,\varphi'(y) \dee y,
\end{aligned}
\end{align}
where we have applied Fatou's lemma in the last transition. Focusing on~\({ (-\uppi, 0) }\), we know that both the first factor in the integrand is strictly positive, \({ N'(\varphi) = c - n'(\varphi) > 0 }\) (because \({ \varphi }\) is strictly increasing), and \({ \varphi' \gneq 0 }\). Thus \({ \varphi' > 0 }\) on~\({ (-\uppi, 0) }\).
\end{proof}

We next start to investigate what happens if solutions touch~\({ \mu }\), and begin with a one-sided \({ \alpha }\)-Hölder estimate around~\({ 0 }\).

\begin{lemma} \label{thm:lower-bound-zero}
Let \({ \varphi \in \mathring{\textnormal{C}}_{\textnormal{even}}(\T) }\) be a nontrivial solution of~\eqref{eq:steady} that is increasing in~\({ (- \uppi, 0) }\) and satisfies \({ n'(\varphi) \leq c }\) everywhere. Then uniformly around~\({ 0 }\) one has
\begin{equation*}
\mu - \varphi(x) \gtrsim \abs{ x }^\alpha.
\end{equation*}
\end{lemma}
\vspace*{-.5em} 
\enlargethispage{1.5\baselineskip} 
\begin{proof}
Let \({ x \in (-\uppi, 0) }\) be close to~\({ 0 }\) and let \({ \xi \in \bigl( x, \tfrac{ x }{ 2 } \bigr)  }\). Monotonicity yields that \({ N'(\varphi(x)) \geq N'(\varphi(\xi)) }\), and since \({ \varphi' > 0 }\) on \({ (-\uppi, 0) }\) by \cref{thm:nodal-pattern}, we may compute
\begin{align*}
N'(\varphi(x)) \, \varphi'(\xi) &\geq N'(\varphi(\xi)) \, \varphi'(\xi) \\[1ex]
& \geq \int_{ - \uppi }^{ 0 } \bigl( K_\alpha(\xi - y) - K_\alpha(\xi + y) \bigr) \, \varphi'(y) \dee y
\end{align*}
using Fatou's lemma as in~\eqref{eq:fatou}. By strict positivity of the integrand and the mean value theorem with \({ \zeta \in (\xi, \xi - 2y) }\), this may be continued as
\begin{align*} 
N'(\varphi(x)) \, \varphi'(\xi) & \geq \int_{ x }^{ \frac{ x }{ 2 } } K_\alpha'( \zeta + y) (-2y) \,\varphi'(y) \dee y \\[1ex]
& \geq \abs{ x } \smashoperator{\min_{ y \in \left[ x, \tfrac{ x }{ 2 } \right] }} K_\alpha'( \zeta + y) \left( \varphi \bigl( \tfrac{ x }{ 2 } \bigr) - \varphi(x) \right) \\[1ex]
& \gtrsim \abs{ x }^{ \alpha - 1} \left( \varphi \bigl( \tfrac{ x }{ 2 } \bigr) - \varphi(x) \right),
\end{align*}
where we have used that \({ \min K_\alpha'( \zeta + y) \geq K_\alpha'(2x) \eqsim \abs{ x }^{ \alpha - 2 } }\) as \({ x \to 0 }\) by \cref{thm:periodic-kernel}. We then integrate over \({ \bigl( x, \tfrac{ x }{ 2 } \bigr)  }\) in \({ \xi }\) and divide by \({  \varphi \bigl( \tfrac{ x }{ 2 } \bigr) - \varphi(x) }\) on both sides, which is valid since \({ \varphi }\) is strictly increasing on~\({ (-\uppi, 0) }\). This gives
\begin{equation*}
N'(\varphi(x)) \gtrsim \abs{ x }^{ \alpha}
\end{equation*}
uniformly around~\({ 0 }\). The stated bound is now a consequence of
\begin{equation} \label{eq:equivalence-lhopital}
N'(\varphi(x)) = c - n'(\varphi(x)) \sim \mu^{p - 1} - (\varphi(x))^{p - 1} \eqsim \mu - \varphi(x),
\end{equation}
where the latter uniform equivalence around~\({ 0 }\) follows from continuity of \({ \varphi }\) and the observation by L'Hôpital's rule that
\begin{equation*}
\lim_{ t \nearrow \mu } \frac{ \mu^{p - 1} - t^{p - 1} }{\mu - t } = (p - 1) \mu^{p - 2} > 0.
\end{equation*}
\end{proof}

\begin{lemma} \label{thm:lower-bound-speed}
The wave speed \({ c }\) is uniformly bounded away from~\({ 0 }\) over the class of solution pairs~\({ (\varphi, c) }\) for which \({ \varphi \in \mathring{\textnormal{C}}_{\textnormal{even}}(\T) }\) is nontrivial, increasing in~\({ (- \uppi, 0) }\), and satisfies \({ n'(\varphi) \leq c }\) everywhere, where we in case~\eqref{eq:n-sgn} also assume that \({ \varphi \bigl( - \tfrac{ \uppi }{ 2 } \bigr) = 0 }\). The estimate \({ c - n'(\varphi(- \uppi)) \gtrsim 1 }\) holds in case~\eqref{eq:n-abs}, implying that \({ \varphi }\) is smooth around~\({ -\uppi }\).	
\end{lemma}

\begin{proof}
Let \({ x \in I \coloneqq \bigl[ -\tfrac{ 3 \uppi }{ 4 }, -\tfrac{ \uppi }{ 4 } \bigr] }\) and consider first case~\eqref{eq:n-abs}. Monotonicity of~\({ N' }\) and~\({ \varphi }\) plus~\eqref{eq:fatou} show that
\begin{align*} \SwapAboveDisplaySkip
N'(\varphi(- \uppi)) \, \varphi'(x) & \geq N'(\varphi(x)) \, \varphi'(x) \\[1ex]
& \geq \int_{ - \uppi }^{ 0 } \bigl( K_\alpha(x - y) - K_\alpha(x + y) \bigr) \, \varphi'(y) \dee y \\[1ex]
& \geq \int_{ I } \bigl( K_\alpha(x - y) - K_\alpha(x + y) \bigr) \, \varphi'(y) \dee y \\[1ex]
& \geq M_\alpha \left( \varphi \bigl( - \tfrac{ \uppi }{ 4 } \bigr) - \varphi \bigl( - \tfrac{ 3 \uppi }{ 4 } \bigr)  \right),
\end{align*}
where we have used that \({ M_\alpha \coloneqq \min \Set[\big]{ K_\alpha(x - y) - K_\alpha(x + y) \given x, y \in I } > 0 }\) by the extreme value theorem and the fact that \({ K_\alpha }\) is even and strictly increasing on \({ (- \uppi, 0) }\) by \cref{thm:K-explicit}. Integrating over~\({ I }\) in~\({ x }\) then yields
\begin{equation} \label{eq:independent-lower-bound-N}
c - n'(\varphi(- \uppi)) = N'(\varphi(- \uppi)) \geq \tfrac{ \uppi }{ 2 } M_\alpha > 0
\end{equation}
after cancelling \({ \varphi \bigl( - \tfrac{ \uppi }{ 4 } \bigr) - \varphi \bigl( - \tfrac{ 3 \uppi }{ 4 } \bigr) > 0 }\) on both sides. Suppose to the contrary that there exists a sequence \({ \Set{ (\varphi_j, c_j) }_j }\) of such solution pairs for which \({ c_j \searrow 0 }\). Then \({ n'(\varphi_j(- \uppi)) \leq c_j \searrow 0 }\) as well, contradicting~\eqref{eq:independent-lower-bound-N}. Thus \({ c \gtrsim 1 }\) uniformly and \({ n'(\varphi(- \uppi)) }\) does not touch~\({ c }\), so \({ \varphi }\) is smooth around~\({ - \uppi }\) by \cref{thm:smoothness}.

In case~\eqref{eq:n-sgn} we consider \({ - \tfrac{ \uppi }{ 2 } }\) instead of~\({ - \uppi }\) and similarly obtain \({ c = N' \bigl( \varphi \bigl( - \tfrac{ \uppi }{ 2 } \bigr)  \bigr) \gtrsim 1 }\).
\end{proof}

Finally we come to the main result in this section, which concerns both the global regularity of solutions and the exact \({ \alpha }\)-Hölder regularity at~\({ 0 }\) for solutions that touch~\({ \mu }\). This is the most technical part of the paper.

\begin{theorem}[Regularity] \label{thm:regularity}
Let \({ \varphi \in \mathring{\textnormal{C}}_{\textnormal{even}}(\T) }\) be a nontrivial solution of~\eqref{eq:steady} that is increasing in~\({ (- \uppi, 0) }\) and satisfies \({ n'(\varphi) \leq c }\), with maximum \({ \varphi(0) = \mu }\). Then
\begin{enumerate}[ref=\roman*)]
\item \label{thm:regularity-smoothness}
\({ \varphi }\) is strictly increasing on~\({ (- \uppi, 0) }\), smooth except at \({ 0 }\) and possibly the point~\({ \varphi^{-1}(0) }\), and features at least the same regularity in the Hölder scale around \({ \varphi^{-1}(0) }\) as \({ n }\) around~\({ 0 }\);
\item \label{thm:regularity-global}
\({ \varphi \in \mathring{\textnormal{C}}_{\textnormal{even}}^\alpha(\T) }\); and
\item \label{thm:regularity-at-zero}
\({ \varphi }\) is exactly \({ \alpha }\)-Hölder continuous at~\({ 0 }\), that is, uniformly around~\({ 0 }\) we have
\begin{equation*}
\mu - \varphi(x) \eqsim \abs{ x }^{ \alpha}.
\end{equation*}
\end{enumerate}
\end{theorem}
\begin{proof}
Property~\ref{thm:regularity-smoothness} and the lower bound in property~\ref{thm:regularity-at-zero} follow directly from \cref{thm:smoothness,thm:nodal-pattern,thm:lower-bound-zero}. As a consequence, it remains to establish global \({ \alpha }\)-Hölder regularity and the upper bound in property~\ref{thm:regularity-at-zero}. Note that Hölder regularity at a point plus smoothness everywhere except at that point does \emph{not} in general imply global Hölder regularity---one additionally needs uniform Hölder regularity \emph{around} the point. In particular, in order to obtain property~\ref{thm:regularity-global}, it suffices to prove \({ \textnormal{C}^\alpha }\) regularity in a small interval around~\({ 0 }\).

To this end, we first establish \({ \textnormal{C}^\beta }\) regularity (around~\({ 0 }\)) for all~\({ \beta < \alpha }\). Let \({ - r \leq y < x \leq 0}\) with \({ 0 < r \ll 1 }\), and observe from Taylor's theorem that
\begin{equation*}
N(\varphi(x)) - N(\varphi(y)) = \left( \varphi(x) - \varphi(y) \right) N'(\varphi(x)) - \tfrac{ 1 }{ 2 } \left( \varphi(x) - \varphi(y) \right)^2 N''(\varphi(\xi))
\end{equation*}
for some \({ \xi \in (y, x) }\) due to the intermediate value theorem. By~\eqref{eq:equivalence-lhopital} we know that
\begin{equation*}
N'(\varphi(x)) \eqsim \mu - \varphi(x),
\end{equation*}
and \({ -N''(\varphi(\xi)) = n''(\varphi(\xi)) \gtrsim 1 }\) independently of~\({ \xi }\) by choosing \({ r }\) so small that \({ \varphi(- r) > 0 }\) and remembering that \({ \varphi }\) is monotone (\({ \varphi(\xi) \geq \varphi(- r) }\)). This gives
\begin{equation*}
N(\varphi(x)) - N(\varphi(y)) \gtrsim \left( \varphi(x) - \varphi(y) \right) \left( \mu - \varphi(x) \right) + \left( \varphi(x) - \varphi(y) \right)^2
\end{equation*}
uniformly over \({ -r \leq y < x \leq 0 }\), and by evenness of~\({ \varphi }\), also uniformly over \({ x, y \in (-r, r) }\) with \({ \abs{ y } > \abs{ x } }\). Thus~\eqref{eq:steady} implies both that
\begin{align}
\abs{ \textnormal{D} }^{- \alpha} \varphi(x) - \abs{ \textnormal{D} }^{- \alpha} \varphi(y) &\gtrsim \left( \varphi(x) - \varphi(y) \right) \left( \mu - \varphi(x) \right) \label{eq:L-difference-highest}
\shortintertext{and}
\abs{ \textnormal{D} }^{- \alpha} \varphi(x) - \abs{ \textnormal{D} }^{- \alpha} \varphi(y) &\gtrsim \left( \varphi(x) - \varphi(y) \right)^2 \label{eq:L-difference-square}
\end{align}
uniformly over \({ x, y \in (-r, r) }\) with \({ \abs{ y } > \abs{ x } }\). Since \({ \abs{ \textnormal{D} }^{- \alpha} }\) is locally \({ \textnormal{C}_\ast^\alpha }\) smoothing by \cref{thm:local-smoothing}, we deduce by a bootstrapping argument in~\eqref{eq:L-difference-square} that \({ \varphi }\), being a~priori only continuous around~\({ 0 }\), is in fact \({ \textnormal{C}^\beta }\)~regular around~\({ 0 }\) for all \({ \beta < \min \Set*{ \tfrac{ 1 }{ 2 }, \alpha } }\). 

If \({ \alpha \in \bigl( \tfrac{ 1 }{ 2 }, 1 \bigr) }\), however, then \({ \abs{ \textnormal{D} }^{- \alpha} \varphi }\) will eventually pass index~\({ 1 }\) in the Hölder scale, and we must instead work with derivatives. Specifically, take any \({ \beta }\) sufficiently close to~\({ \tfrac{ 1 }{ 2 } }\) from the previous argument such that now \({ \alpha + \beta > 1 }\) and \({ \abs{ \textnormal{D} }^{- \alpha} \varphi \in \mathring{ \textnormal{C} }_\ast^{ \alpha + \beta }(\T) = \mathring{ \textnormal{C}}^{1, \alpha + \beta - 1}(\T) }\). Then, since \({ \abs{ \textnormal{D} }^{- \alpha} \varphi'(0) = 0 }\), we find from the mean value theorem that
\begin{align} \label{eq:L-difference-derivative}
\begin{aligned}
\abs{ \textnormal{D} }^{- \alpha} \varphi(x) - \abs{ \textnormal{D} }^{- \alpha} \varphi(y) &= \abs{ x - y } \, \abs{ \abs{ \textnormal{D} }^{- \alpha} \varphi'(\xi) - \abs{ \textnormal{D} }^{- \alpha} \varphi'(0) } \\[1ex]
&\lesssim \abs{ x - y } \, \abs{ \xi }^{ \alpha + \beta - 1} \\[1ex]
&< \abs{ x - y } \, \abs{ y }^{ \alpha + \beta - 1}
\end{aligned}
\end{align}
for some \({ \xi \in (y, x) }\), where \({ - r \leq y < x \leq 0}\), as above. If \({ \abs{ x } < \abs{ x - y } }\), then~\eqref{eq:L-difference-square}, \eqref{eq:L-difference-derivative}, and the triangle inequality imply that
\begin{equation} \label{eq:holder-estimate-first-case}
\varphi(x) - \varphi(y) \lesssim \abs{ x - y }^{ \frac{ \alpha + \beta }{ 2 } }.
\end{equation}
In particular, for \({ x = 0 }\), this gives
\begin{equation} \label{eq:holder-estimate-in-y}
\mu - \varphi(y) \lesssim \abs{ y }^{ \frac{ \alpha + \beta }{ 2 } }
\end{equation}
for all~\({ y \in [- r, 0] }\). Otherwise, if \({ \abs{ x - y } \leq \abs{ x } }\), then~\eqref{eq:L-difference-highest}, \eqref{eq:L-difference-derivative}, and the triangle inequality yield
\begin{equation*}
\left( \varphi(x) - \varphi(y) \right) \left( \mu - \varphi(x) \right) \lesssim \abs{ x - y } \, \abs{ x }^{ \alpha + \beta - 1 }.
\end{equation*}
We next use that \({ \mu - \varphi(x) \gtrsim \abs{ x }^{ \alpha} }\) by \cref{thm:lower-bound-zero} and get
\begin{equation} \label{eq:holder-estimate-intermediate-step}
\varphi(x) - \varphi(y) \lesssim \frac{ \abs{ x - y } }{ \abs{ x }^{ 1 - \beta } }.
\end{equation}
Interpolating between \eqref{eq:holder-estimate-in-y} and \eqref{eq:holder-estimate-intermediate-step} with index~\({ \gamma \in (0, 1) }\), and using that \({ \abs{ y } \leq 2 \abs{ x } }\), then subsequently show that
\begin{align*} \SwapAboveDisplaySkip
\frac{ \varphi(x) - \varphi(y) }{ \abs{ x - y }^{ \gamma }  } &\leq \frac{ \left( \varphi(x) - \varphi(y) \right)^\gamma }{ \abs{ x - y }^{ \gamma } } \, \left( \mu - \varphi(y) \right)^{ 1 - \gamma } \\[1ex]
& \lesssim \abs{ x }^{ (\beta - 1) \gamma + \frac{ \alpha + \beta }{ 2 }(1 - \gamma)}
\end{align*}
is uniformly bounded over \({ - r \leq y < x \leq 0 }\) in the case \({ \abs{ x - y } \leq \abs{ x } }\) provided the last exponent is nonnegative, that is, if
\begin{equation*}
\gamma \leq \frac{ \alpha + \beta}{ 2 + \alpha - \beta}.
\end{equation*}
As such, by choosing the maximal~\({ \gamma }\), we obtain the estimate
\begin{equation*}
\varphi(x) - \varphi(y) \lesssim \abs{ x - y }^{ (\alpha + \beta) / ( 2 + \alpha - \beta) }
\end{equation*}
when \({ \abs{ x - y } \leq \abs{ x } }\), so that, together with~\eqref{eq:holder-estimate-first-case} it is true that
\begin{equation*}
\varphi(x) - \varphi(y) \lesssim \max \Set[\big]{ \abs{ x - y }^{ \frac{ \alpha + \beta }{ 2 } },  \abs{ x - y }^{ (\alpha + \beta) / ( 2 + \alpha - \beta) } }
\end{equation*}
uniformly over~\({ - r \leq y < x \leq 0 }\). Since both exponents are strictly increasing in \({ \beta }\) and converge to~\({ \alpha }\) as \({ \beta \nearrow \alpha }\), it follows by bootstrapping that \({ \varphi }\) is \({ \textnormal{C}^\beta }\)~regular around~\({ 0 }\) for all~\({ \beta < \alpha }\).

We next establish the upper \({ \textnormal{C}^\alpha }\) estimate at~\({ 0 }\) in property~\ref{thm:regularity-at-zero}. In fact, with \({ u(x) \coloneqq \mu - \varphi(x) }\), we shall prove that
\begin{equation} \label{eq:holder-highest-difference-upper}
u(x) \lesssim \abs{ x }^{ \beta }
\end{equation}
uniformly over \({ x \in (-r, r) }\) and \({ \beta \in [0, \alpha) }\), from which the desired estimate follows by letting \({ \beta \nearrow \alpha }\). On this route, note from~\eqref{eq:L-difference-square} that
\begin{align} \label{eq:holder-highest-integral-difference}
\begin{aligned}
(u(x))^2 \lesssim \abs{ \textnormal{D} }^{- \alpha} \varphi(0) - \abs{ \textnormal{D} }^{- \alpha} \varphi(x) &= \hphantom{\tfrac{ 1 }{ 2 }}\int_{ \T } \bigl( K_\alpha(y) - K_\alpha(x - y) \bigr) \, \varphi(y) \dee y \\[1ex]
&= \hphantom{\tfrac{ 1 }{ 2 }} \int_{ \T } \bigl( K_\alpha(x - y) - K_\alpha(y) \bigr) \, u(y) \dee y \\[1ex]
&= \tfrac{ 1 }{ 2 } \int_{ \T } \diamondsuit_y K_\alpha(x) \, u(y) \dee y,
\end{aligned}
\end{align}
where \({ \diamondsuit_y f(x) \coloneqq f(x + y) - 2f(y) + f(x - y) }\) denotes the second-order central difference operator. Here we have utilised periodicity of~\({ K_\alpha }\) in the first transition between the integrals, and averaging, variable change \({ y \mapsto -y }\), and evenness of~\({ u }\) (from~\({ \varphi }\)) in the last step. Since we have already established that \({ u \in \textnormal{C}^{ \beta}(\T) }\) for all \({ \beta \in [0, \alpha) }\), it is clear that \({ u \abs{  }^{ -\beta } }\) is bounded on~\({ (- r, r) }\). Thus~\eqref{eq:holder-highest-integral-difference} shows that
\begin{equation*}
\sup_{ \abs{ x } < r } \abs*{u(x) \abs{ x }^{ - \beta } } \lesssim \sup_{ \abs{ x } < r } \, \abs{ x }^{ -2 \beta } \int_{ \T } \abs[\big]{ \diamondsuit_y K_\alpha(x) } \, \abs{ y }^{ \beta } \dee y
\end{equation*}
after cancelling \({ \sup_{ \abs{ x } < r } \abs*{u(x) \abs{ x }^{ - \beta } } }\) once on each side. Now remember that \({ K_\alpha = \gamma_\alpha \abs{  }^{ \alpha - 1 } + K_{ \alpha, \textnormal{reg}} }\) from \cref{thm:periodic-kernel}. In particular, for the regular part we may Taylor expand around~\({ y }\) to see that
\begin{equation*}
\abs[\big]{ \diamondsuit_y K_{ \alpha, \textnormal{reg}}(x)} = \mathcal{O}(x^2)
\end{equation*}
uniformly over~\({ y \in \T }\), because \({ K_{ \alpha, \textnormal{reg}} }\) is even and \({ K''_{ \alpha, \textnormal{reg}} }\) and \({ K'''_{ \alpha, \textnormal{reg}} }\) are bounded on~\({ \T }\). As such, using that \({ \abs{ y }^{ \beta } \lesssim 1 }\) on~\({ \T }\) independently of~\({ \beta }\), we obtain that
\begin{equation*}
\abs{ x }^{ -2 \beta } \int_{ \T } \abs[\big]{ \diamondsuit_y K_{ \alpha, \textnormal{reg}}(x)} \, \abs{ y }^{ \beta } \dee y = \mathcal{O} \bigl( \abs{ x }^{ 2(1 - \beta) }  \bigr) = \mathcal{O}(1)
\end{equation*}
since \({ \beta < 1 }\). For the singular part, one has with \({ y = xs }\) that
\begin{equation*}
\abs{ x }^{ -2 \beta } \smashoperator{\int\limits_{ \abs{ y }\leq \uppi }} \abs[\big]{ \diamondsuit_y \abs{ }^{ \alpha - 1 }(x) } \, \abs{ y }^{ \beta } \dee y \leq \abs{ x }^{ \alpha - \beta } \smashoperator{\int\limits_{ \abs{ s } < \infty }} \abs[\big]{ \diamondsuit_1 \abs{ }^{ \alpha - 1 }(s) } \, \abs{ s }^{ \beta } \dee s \qquad \text{\footnotesize(note the \({ 1 }\) in \({ \diamondsuit_1 }\)).}
\end{equation*}
The right-hand side is \({ \mathcal{O}(1) }\) over~\({ x \in (- r, r) }\) because of \({ \alpha - \beta \geq 0 }\) and the following observation: \({ \abs{  }^{ \alpha - 1 } }\)~is locally integrable, and
\begin{equation*}
\diamondsuit_1 \abs{ }^{ \alpha - 1 }(s) = \abs{ s }^{ \alpha - 1 } \bigl[ (\alpha - 1) (\alpha - 2) \, s^{-2} + \mathcal{O}(s^{-4}) \bigr]
\end{equation*}
as~\({ \abs{ s } \to \infty }\), so that
\begin{equation*}
\abs[\big]{ \diamondsuit_1 \abs{ }^{ \alpha - 1 }(s) } \, \abs{ s }^{ \beta } \lesssim \abs{ s }^{ \alpha + \beta - 3 }
\end{equation*}
for \({ \abs{ s } \gg 1 }\), where \({ \alpha + \beta - 3 < -1 }\) uniformly over \({ \beta < \alpha }\) because \({ \alpha \in (0, 1) }\) is fixed, thereby guaranteeing integrability at infinity. Hence, \({ \sup_{ \abs{ x } < r } \abs[\big]{u(x) \abs{ x }^{ - \beta } } \lesssim 1 }\) uniformly over \({ \beta  < \alpha }\), which is~\eqref{eq:holder-highest-difference-upper}.

It remains to establish \({ \textnormal{C}^\alpha }\) continuity around~\({ 0 }\). Since \({ \varphi }\) is increasing on \({ [-\uppi, 0] }\) and even, it suffices to show that
\begin{equation*}
\sup_{ \substack{x \in [- r, 0); \\ h \in ( 0, \abs{ x }]} } \frac{ \upDelta_{h}\varphi(x) }{ h^\alpha } < \infty,
\end{equation*}
where we have introduced the (scaled) symmetric difference \({ \upDelta_h f(x) \coloneqq f(x + h) - f(x - h) }\). To this end, we shall extract \({ \upDelta_{h}\varphi(x) }\) from \({ \upDelta_h[N(\varphi)](x) }\) and estimate each side of the relation
\begin{equation} \label{eq:steady-fristorder-difference}
\upDelta_h[N(\varphi)](x) = \upDelta_{h}[\abs{ \textnormal{D} }^{- \alpha}\varphi](x),
\end{equation}
which comes straight from~\eqref{eq:steady}. On this path, we let \({ x \in [- r, 0) }\) and \({ h \in (0, \abs{ x }] }\), and then choose \({ a \coloneqq \varphi(x + h) }\) and \({ b \coloneqq \varphi(x - h) }\) in the Taylor expansion
\begin{equation*}
N(b) = N(a) + N'(a)(b - a) + \tfrac{ 1 }{ 2 }N''(\zeta) (b - a)^2,
\end{equation*}
with \({ \zeta }\) between \({ a }\) and~\({ b }\), to see that
\begin{equation} \label{eq:firstorder-difference-N}
\upDelta_h[N(\varphi)](x) = \left(N'(\varphi(x + h)) - \tfrac{ 1 }{ 2 }N''(\varphi(\xi)) \, \upDelta_h \varphi(x) \right) \upDelta_h \varphi(x)
\end{equation}
for some \({ \xi \in (x -h, x + h) }\) (satisfying \({ \varphi(\xi) = \zeta }\)) by the intermediate value theorem. Here
\begin{equation*}
 -N''(\varphi(\xi)) = n''(\varphi(\xi)) \eqsim 1
\end{equation*}
uniformly over \({ x \in [- r, 0) }\) and~\({ h \in (0, \abs{ x }] }\). Now note that
\begin{align*}
N'(\varphi(x + h)) - p \left( \varphi(x)^{p - 1} - \varphi(x + h)^{p - 1} \right) \hspace{11em} \\[1ex]
\hspace{11em} \sim \mu^{p - 1} - \varphi(x)^{p - 1} \eqsim \mu - \varphi(x) \eqsim \abs{ x }^{ \alpha }
\end{align*}
in light of \eqref{eq:equivalence-lhopital} and the exact \({ \textnormal{C}^\alpha }\)~estimate at~\({ x = 0 }\). Since \({ \varphi(x)^{p - 1} - \varphi(x + h)^{p - 1} }\) and \({ \upDelta_h \varphi(x) }\) both vanish as \({ h \searrow 0 }\), we see that
\begin{equation*}
\sup_{ h \in ( 0, \abs{ x }] } \left(N'(\varphi(x + h)) - \tfrac{ 1 }{ 2 }N''(\varphi(\xi)) \, \upDelta_{h} \varphi(x) \right) \gtrsim \abs{ x }^{ \alpha }.
\end{equation*}
Thus \eqref{eq:firstorder-difference-N} yields that
\begin{equation} \label{eq:firstorder-difference-phi-and-N}
\smashoperator{\sup_{ h \in ( 0, \abs{ x }] }} \frac{ \upDelta_{h}\varphi(x) }{ h^\beta } \lesssim \abs{ x }^{ - \alpha } \smashoperator{\sup_{ h \in ( 0, \abs{ x }] }} \frac{ \upDelta_{h}[N(\varphi)](x) }{ h^\beta }
\end{equation}
for all \({ \beta < \alpha }\), where we postpone taking the supremum over~\({ x \in [- r, 0) }\) until we have estimates for \({ \upDelta_{h}[\abs{ \textnormal{D} }^{- \alpha} \varphi](x) }\) in~\eqref{eq:steady-fristorder-difference}. With that in mind, we first consider the regular part in~\({ \abs{ \textnormal{D} }^{- \alpha} \varphi }\) and compute
\begin{align*}
\upDelta_{h} \bigl[K_{ \alpha, \textnormal{reg}} \ast \varphi\bigr](x) &= \hphantom{xh} \int_{ \T} \upDelta_{h} K_{ \alpha, \textnormal{reg}}(x - y) \, \varphi(y) \dee y \\[1ex]
&= \hphantom{x} h \int_{ \T } 2 \, K_{ \alpha, \textnormal{reg}}'(x - y) \, \varphi(y) \dee y \\[1ex]
&= \hphantom{x} h \int_{ \T } \upDelta_{x} K_{ \alpha, \textnormal{reg}}' (y) \, \varphi(y) \dee y \qquad \text{\footnotesize(note the \({ x }\) in \({ \upDelta_x }\))} \\[1ex]
&= xh \int_{ \T } \int_{ -1 }^{ 1 } K_{ \alpha, \textnormal{reg}}''(y + tx) \dee t \, \varphi(y) \dee y
\end{align*}
by the mean value theorem and repeated use of parity and periodicity of \({ K_{ \alpha, \textnormal{reg}} }\) and~\({ \varphi }\). Consequently,
\begin{equation*}
\abs[\big]{\upDelta_{h}\bigl[K_{ \alpha, \textnormal{reg}} \ast \varphi \bigr](x)} \lesssim \norm{ \varphi }_{\mathring{ \textnormal{C}}^\beta(\T) }^{ \theta } \abs{ x } h < \norm{ \varphi }_{\mathring{ \textnormal{C}}^\beta(\T) }^{ \theta } \abs{ x }^{ \alpha } h^\beta,
\end{equation*}
for any \({ \theta \in (0, 1) }\) because \({ K_{ \alpha, \textnormal{reg}}'' }\) is bounded and \({ \norm{ \varphi }_{ \infty } \lesssim \norm{ \varphi }_{\mathring{ \textnormal{C}}^\beta(\T) }^{ \theta } \norm{ \varphi }_{ \infty }^{ 1 - \theta } \lesssim \norm{ \varphi }_{\mathring{ \textnormal{C}}^\beta(\T) }^{ \theta } }\). Hence,
\begin{equation} \label{eq:firstorder-difference-regular-final}
\abs{ x }^{ - \alpha } \smashoperator{\sup_{ h \in ( 0, \abs{ x }] }} \frac{ \upDelta_{h}\bigl[K_{ \alpha, \textnormal{reg}} \ast \varphi \bigr](x) }{ h^\beta } \lesssim \norm{ \varphi }_{\mathring{ \textnormal{C}}^\beta(\T) }^{ \theta }.
\end{equation}
Switching to the singular part, one finds by parity and periodicity that
\begin{align} \label{eq:firstorder-difference-singular}
\begin{aligned}
\upDelta_{h}\bigl[ \abs{  }^{ \alpha - 1 } \ast \varphi \bigr](x) &= \hphantom{h^\alpha} \smashoperator{\int\limits_{ -\uppi }^{ 0 }} \upDelta_{h}\abs{  }^{ \alpha - 1 }(y) \, \upDelta_{\abs{ x }} \varphi(y) \dee y \\[1ex]
&= \smash{h^\alpha \smashoperator{\int\limits_{ -\uppi / h }^{ 0 }} \upDelta_{1}\abs{  }^{ \alpha - 1 }(s) \, \upDelta_{\abs{ x }}\varphi(hs) \dee s } \vphantom{\int\limits^0} \qquad \text{\footnotesize(note the subscripts).} 
\end{aligned}
\end{align}
Since \({ \varphi \in \mathring{ \textnormal{C}}^\beta(\T) }\), we have
\begin{align}
&\abs{  \upDelta_{\abs{ x }}\varphi(y) } \lesssim \norm{ \varphi}_{\mathring{ \textnormal{C}}^\beta(\T)} \min \Set{ \abs{ x }^\beta, \abs{ y }^\beta } \quad \textnormal{for } \beta < \alpha, \label{eq:delta-estimate-beta}
\shortintertext{and furthermore,}
&\abs{ \upDelta_{\abs{ x }}\varphi(y) } \lesssim \max \Set{ \abs{ x }^\alpha, \abs{ y }^\alpha } \label{eq:delta-estimate-alpha}
\end{align}
by the already established estimate \({ \mu - \varphi( \xi) \lesssim \abs{ \xi }^{ \alpha } }\) for \({ \abs{ \xi } \ll 1 }\). Interpolating between~\eqref{eq:delta-estimate-beta} and~\eqref{eq:delta-estimate-alpha} with parameter
\begin{equation*}
\theta \coloneqq \frac{ \alpha }{ \alpha + \beta } \in \left( \tfrac{ 1 }{ 2 }, 1 \right), \quad \textnormal{so that} \quad \theta \beta = (1 - \theta) \alpha,
\end{equation*}
then yields
\begin{align} \SwapAboveDisplaySkip \label{eq:delta-estimate-holder-small}
\begin{aligned}
\abs{ \upDelta_{\abs{ x }}\varphi(y) } &\lesssim \norm{ \varphi}_{\mathring{ \textnormal{C}}^\beta(\T)}^{ \theta} \min \Set*{ \abs{ x }^{ \theta \beta}, \abs{ y }^{ \theta\beta} } \max \Set*{ \abs{ x }^{ (1 - \theta) \alpha}, \abs{ y }^{ (1 - \theta) \alpha} } \\[1ex]
&=  \norm{ \varphi}_{\mathring{ \textnormal{C}}^\beta(\T)}^{ \theta} \abs{ x y }^{ \theta \beta }.
\end{aligned}
\end{align}
This estimate, with \({ y = hs }\), is appropriate for small~\({ s }\) in~\eqref{eq:firstorder-difference-singular}, but becomes problematic for large~\({ s }\) when \({ \alpha > 2 / 3 }\) since
\begin{equation*}
\upDelta_{1}\abs{  }^{ \alpha - 1 }(s) \, \abs{ s }^{ \theta \beta } \eqsim \abs{ s }^{ \alpha - 2 + \theta \beta }
\end{equation*}
for \({ \abs{ s } \gg 1 }\) (at scale \({ s \sim h^{-1} }\)), thus failing to be integrable in~\eqref{eq:firstorder-difference-singular} as~\({ h \searrow 0 }\). As a remedy, we use the estimate
\begin{align} \SwapAboveDisplaySkip \label{eq:delta-estimate-holder-large}
\begin{aligned}
\abs*{ \upDelta_{\abs{ x }}\varphi(hs) } & \leq \norm{ \varphi}_{\mathring{ \textnormal{C}}^\beta(\T)}^{ \theta} \abs{ x }^{ \theta \beta } \abs*{ \upDelta_{\abs{ x }}\varphi(hs) }^{1 - \theta} \\[1ex]
& \lesssim \norm{ \varphi}_{\mathring{ \textnormal{C}}^\beta(\T)}^{ \theta} \abs{ x }^{ \theta \beta + 1 - \theta } \smashoperator[l]{\max_{ \abs{ t - hs } \leq \abs{ x } }} \abs{ \varphi'(t) }^{1 - \theta}
\end{aligned}
\end{align}
when~\({ s \sim h^{-1} }\), where one observes that the given maximum of \({ \varphi'(t) }\) is uniformly bounded over \({ h \in (0, \abs{ x }] }\) and~\({ x \in [- r, 0) }\) since \({ t }\) stays away from the singularity at~\({ 0 }\). We then note that
\begin{equation*}
\min \Set[\Big]{ \abs{ h s }^{ \theta \beta}, \abs{ x }^{ 1 - \theta } \smashoperator[l]{\max_{ \abs{ t - hs } \leq \abs{ x } }} \abs{ \varphi'(t) }^{1 - \theta} } \lesssim \max \Set*{ h^{ \theta\beta}, \abs{ x }^{ 1 - \theta }  } \leq \abs{ x }^{ \min \Set{ \theta \beta, 1- \theta } }
\end{equation*}
uniformly over \({ s \in (-\uppi/h, 0) }\), where we have utilised that \({ h \leq \abs{ x } }\). In particular, combining~\eqref{eq:delta-estimate-holder-small} and~\eqref{eq:delta-estimate-holder-large} implies that
\begin{equation*}
\abs{\upDelta_{\abs{ x }}\varphi(hs) } \lesssim \norm{ \varphi}_{\mathring{ \textnormal{C}}^\beta(\T)}^{ \theta} \abs{ x }^{ \theta \beta } \abs{ x }^{ \min \Set{ \theta \beta, 1- \theta } }
\end{equation*}
uniformly over \({ s \in (-\uppi/h, 0) }\). Now~\eqref{eq:firstorder-difference-singular} may be estimated as\pagebreak 
\begin{align*}
\abs[\big]{\upDelta_{h}\bigl[ \abs{  }^{ \alpha - 1 } \ast \varphi \bigr](x)} &\lesssim \norm{ \varphi}_{\mathring{ \textnormal{C}}^\beta(\T)}^{ \theta} h^\alpha \abs{ x }^{ \theta \beta + \min \Set{ \theta \beta, 1- \theta } } \smashoperator{\int\limits_{ - \infty }^{ 0 }} \abs[\big]{\upDelta_{1}\abs{  }^{ \alpha - 1 }(s)} \dee s \\[1ex]
& \lesssim \norm{ \varphi}_{\mathring{ \textnormal{C}}^\beta(\T)}^{ \theta} h^\alpha \abs{ x }^{ \theta \beta + \min \Set{ \theta \beta, 1- \theta } },
\end{align*}
where the integral converges because \({ \upDelta_{1}\abs{  }^{ \alpha - 1 }(s) \lesssim  \abs{ s }^{ \alpha - 2 } }\) for \({ \abs{ s } \gg 1 }\) with \({ \alpha - 2 < -1 }\). Therefore, as~\({ h \leq \abs{ x } }\),
\begin{align} \SwapAboveDisplaySkip \label{eq:firstorder-difference-singular-final}
\begin{aligned}
\abs{ x }^{ - \alpha } \smashoperator{\sup_{ h \in ( 0, \abs{ x }] }} \frac{ \upDelta_{h}\bigl[ \abs{  }^{ \alpha - 1 } \ast \varphi \bigr](x) }{ h^\beta } &\lesssim \norm{ \varphi }_{\mathring{ \textnormal{C}}^\beta(\T) }^{ \theta } \abs{ x }^{ \min \Set*{ (2 \theta - 1) \beta, (1 - \theta)(1 - \beta) \vphantom{\sum} } } \\[1ex]
& \lesssim  \norm{ \varphi }_{\mathring{ \textnormal{C}}^\beta(\T) }^{ \theta }
\end{aligned}
\end{align}
uniformly over \({ x \in [- r, 0) }\) and all \({ \beta < \alpha }\) sufficiently close to~\({ \alpha }\), since in that case
\begin{equation*}
\min \Set[\big]{ (2 \theta - 1) \beta, (1 - \theta)(1 - \beta) } > 0 \quad \textnormal{(uniformly).}
\end{equation*}
We then put \eqref{eq:steady-fristorder-difference}, \eqref{eq:firstorder-difference-phi-and-N}, \eqref{eq:firstorder-difference-regular-final}, and~\eqref{eq:firstorder-difference-singular-final} together and find that
\begin{equation} \label{eq:holder-norm-around-zero}
\sup_{ \substack{x \in [- r, 0); \\ h \in ( 0, \abs{ x }]} } \!\!\! \frac{ \upDelta_{h}\varphi(x) }{ h^\beta } \lesssim \norm{ \varphi }_{\mathring{ \textnormal{C}}^\beta(\T) }^{ \theta }
\end{equation}
uniformly over all \({ \beta }\) sufficiently close to~\({ \alpha }\). By smoothness away from~\({ 0 }\), one has
\begin{equation*}
\smash{\norm{ \varphi }_{\mathring{ \textnormal{C}}^\beta(\T) } \lesssim \max \Set[\Bigg]{ 1,\hspace{.5em} \smashoperator{\sup_{ \substack{x \in [- r, 0); \\ h \in ( 0, \abs{ x }]} }} \frac{ \upDelta_{h}\varphi(x) }{ h^\beta } } } \vphantom{\Bigg\{}
\end{equation*}
for all~\({ \beta \leq \alpha }\), and so~\eqref{eq:holder-norm-around-zero} implies that
\begin{equation*}
\Biggl( \smashoperator[r]{\sup_{ \substack{x \in [- r, 0); \\ h \in ( 0, \abs{ x }]} }} \frac{ \upDelta_{h}\varphi(x) }{ h^\beta } \Biggr)^{1 - \theta} \lesssim 1
\end{equation*}
uniformly over \({ \beta }\) sufficiently close to~\({ \alpha }\). In particular, letting \({ \beta \nearrow \alpha }\) (for which \({ \theta \searrow \tfrac{ 1 }{ 2 } }\) stays away from~\({ 1 }\)), it follows that \({ \varphi }\) is indeed \({ \textnormal{C}^\alpha }\)~continuous around~\({ 0 }\).
\end{proof}
\vspace*{-.8em} 

In case~\eqref{eq:n-sgn} we could have assumed that \({ \varphi(- \uppi) = - \mu }\) instead of \({ \varphi(0) = \mu }\) in \cref{thm:regularity} and then proved exact \({ \alpha }\)-Hölder continuity at~\({ - \uppi }\). We conjecture that both assumptions imply the other and more generally imply antisymmetry of waves about~\({ - \frac{ \uppi }{ 2 } }\) whenever one deals with antisymmetric nonlinearities. This is also the reason why we assumed that \({ \varphi \bigl( - \tfrac{ \uppi }{ 2 } \bigr) = 0 }\) in \cref{thm:lower-bound-speed}. As a remedy to the lack of proof of the general property, we shall in \cref{sec:bifurcation} instead \emph{construct} solutions which are antisymmetric about~\({ - \tfrac{ \uppi }{ 2 } }\).

\enlargethispage{1.1\baselineskip} 
\vspace*{-.8em} 

\section{Global bifurcation analysis} \label{sec:bifurcation}

We first establish nontrivial small-amplitude travelling waves around the line \({ c \mapsto (0, c) }\) of trivial solutions by means of local bifurcation theory and then extend the bifurcation curve globally using the analytic theory of Buffoni and Toland~\autocite{BufTol2003a}. By carefully examining the structure of the global curve in connection with the a~priori nodal properties in \cref{sec:regularity}, we finally deduce the existence of a limiting sequence along the curve which converges to a highest wave satisfying \cref{thm:regularity}. This establishes \cref{thm:existence} when the nonlinearities~\eqref{eq:n} are smooth, that is, when they equal \({ n(x) = x^p }\) for~\({ 2 \leq p \in \N }\), and in \cref{fig:bifurcation} we provide a sketch of the analysis.

\begin{figure}[h!]
\floatbox[{\capbeside\thisfloatsetup{capbesideposition={right,center},capbesidesep=quad}}]{figure}[\FBwidth]
{\caption{Illustrating the global bifurcation diagram in the smooth case \({ n(x) = x^p }\) for \({ 2 \leq p \in \N }\) of \({ 2 \uppi/ k }\)-periodic even solutions obtained in \cref{thm:global-bifurcation} bifurcating from the trivial solution~\({ (0, k^{- \alpha}) }\) and reaching a limiting extreme wave. The dashed vertical lines mark the bounds for the wave speed in \cref{thm:bound-wavespeed,thm:lower-bound-speed-bifurcation}, whereas the solid curve displays the possible maximal height from~\eqref{eq:highest-value} for these waves (plotted for~\({ p = 3 }\)). Along the dotted bifurcation curve, one may extract a sequence for which possibilities~\ref{thm:global-bifurcation-blowup} and~\ref{thm:global-bifurcation-boundary} in \cref{thm:global-bifurcation} occur simultaneously, converging to a solution of~\eqref{eq:steady} with the \({ \textnormal{C}^\alpha }\)~properties of~\cref{thm:regularity}.}\label{fig:bifurcation}}
{%
\begin{tikzpicture}[scale=1.2]
\fill[artcolor-boxbg, domain=.4:4, variable=\x] (.4, 0) --  plot (\x, {3*sqrt(\x/3)^(2/3)}) -- (4, 0) -- cycle;
\draw[thick,->] (-.3,0) -- (4.8,0) node[anchor=west] {\({ c }\)};
\draw[thick,->] (0,-.3) -- (0,4) node[anchor=south] {\({ \max \varphi }\)};
\draw[thick,dashed,color=artcolor] (.4,3.8) -- (.4,0);
\draw[thick,dashed,color=artcolor] (4,3.8) -- (4,0) node[anchor=north,color=artcolortext] { \({ \frac{ p }{ p - 1 }  \norm{K_\alpha}_{ \textnormal{L}^1(\T)} }\)};
\draw[thick] (4,.08) -- (4,-.08);
\draw[domain=0:4.4,smooth,samples=200,thick,color=artcolor,variable=\x] plot (\x, {3*sqrt(\x/3)^(2/3)});
\coordinate (trivial) at (1.7,0);
\coordinate (max) at (3,3);
\draw[thick,dotted] (trivial) node[anchor=north] {\({ k^{- \alpha } }\)} to [in=280,out=90] (max);
\draw[color=artcolortext,fill=artcolortext] (trivial) circle (.05);
\draw[color=artcolortext,fill=artcolortext] (max) circle (.05) node[anchor=south,rotate=20,xshift=-20] {\({ \max \varphi = (c / p)^{1/(p - 1)} }\)};;
\end{tikzpicture}%
}
\end{figure}

In the general nonsmooth situation, however, one cannot use the analytic bifurcation theory directly. We resolve this issue by regularising \({ n }\) analytically around~\({ 0 }\) (where its regularity is only of order~\({ p }\) in the Hölder scale) and instead study global bifurcation for the regularised equation
\begin{equation} \label{eq:steady-reg}
0 = F^\epsilon(\varphi, c) \coloneqq \abs{ \textnormal{D} }^{- \alpha} \varphi - N^\epsilon(\varphi; c) - \textstyle\fint_\T n^\epsilon(\varphi)
\end{equation}
of~\eqref{eq:steady} for every~\({ 0 < \epsilon \ll 1 }\). This leads to solutions~\({ (\varphi^\epsilon, c^\epsilon) }\) at the end of the bifurcation curves, with the optimal \({ \alpha }\)-Hölder continuity of \cref{sec:regularity}, that will be shown to converge (up to a subsequence) to a solution of~\eqref{eq:steady} with the same Hölder properties as~\({ \epsilon \searrow 0 }\). Here \({ N^\epsilon(\varphi; c) \coloneqq c \varphi - n^\epsilon(\varphi) }\) and
\begin{equation} \label{eq:n-reg}
n^{ \epsilon }(x) \coloneqq
\begin{cases}
\bigl( x^2 + \epsilon^2 \bigr)^{p / 2} - \epsilon^p & \text{in case~\eqref{eq:n-abs};} \\[1ex]
x \left( \bigl( x^2 + \epsilon^2 \bigr)^{ (p - 1) / 2} - \epsilon^{p - 1} \right) & \text{in case~\eqref{eq:n-sgn}}
\end{cases}
\end{equation}
is a natural analytic regularisation with the same monotonicity properties as~\({ n }\) and that converges uniformly to~\({ n }\) on compact sets as~\({ \epsilon \searrow 0 }\). In particular, the regularity theory in \cref{sec:regularity} carries over to the new setting by replacing~\({ n }\) and~\({ N }\) with~\({ n^\epsilon }\) and~\({ N^\epsilon }\), noting that the extreme value corresponding to the first positive critical point for~\({ N^\epsilon }\) is a continuous function
\begin{equation} \label{eq:highest-value-reg}
\mu^\epsilon \coloneqq \mu(p, c, \epsilon)
\end{equation}
that converges to~\({ \mu }\) in~\eqref{eq:highest-value} as~\({ \epsilon \searrow 0 }\) by the implicit function theorem.

In the remainder, we focus on the analysis of the nonsmooth situation, leaving the appropriate modifications (\enquote{\({ \epsilon =  0}\)}) when \({ n(x) = x^p }\) for \({ 2 \leq p \in \N }\) to the reader, but shall nevertheless provide details for the bifurcation formulas in the smooth case as they may be of independent interest.

According to the above, we study~\({ F^\epsilon }\) from~\eqref{eq:steady-reg} as an operator \({ \mathcal{X}^{ \beta } \times  \R_+ \to \mathcal{X}^{ \beta } }\), where \({ \R_+ \coloneqq [0, \infty) }\) and \({ \mathcal{X}_{\vphantom{even}}^{\beta} \coloneqq  \mathring{ \textnormal{C}}_{ \textnormal{even}}^\beta(\T) }\), noting that \({ N^\epsilon(\cdot, c) }\) acts on~\({ \mathcal{X}^{ \beta} }\) in light of~\autocite[Theorem~2.1]{GoeSac1999a}. We also let \({ \beta \in \bigl(\max \Set*{ \alpha, \tfrac{ 1 }{ 2 } }, 1 \bigr) }\); the choice \({ \beta > \frac{ 1 }{ 2 } }\) guarantees that the Fourier series of \({ \varphi \in \mathcal{X}^\beta }\) converges uniformly to~\({ \varphi }\), whereas the requirement \({ \beta \in (\alpha, 1) }\) avoids the technicalities of the Hölder--Zygmund space of order~\({ 1 }\) and assures that \({ \mathcal{X}^\beta }\) contains the sought-after extreme wave in \cref{thm:regularity}.

Observe that \({ F^\epsilon }\) is analytic due to the regularisation and that its linearisation around the line of trivial solutions equals
\begin{equation*}
\partial_{ \varphi} F^\epsilon(0, c) =  \abs{ \textnormal{D} }^{- \alpha} - c \identityoperator.
\end{equation*}
Hence, for \({ c > 0 }\) the kernel of \({ \partial_{ \varphi} F^\epsilon(0, c) }\) is trivial unless \({ c = k^{- \alpha} }\) for some integer~\({ k \geq 1}\), being a simple eigenvalue of~\({ \abs{ \textnormal{D} }^{- \alpha} }\), in which case
\begin{equation*} 
\ker \partial_{ \varphi} F^\epsilon(0, k^{ - \alpha } ) = \vectorspan_{k \geq 1} \Set{ \cos (k \cdot) }
\end{equation*}
is one-dimensional. Furthermore, \({ \abs{ \textnormal{D} }^{- \alpha} }\) is a compact operator \({ \mathcal{X}^\beta \to \mathcal{X}^\beta }\) since it is \({ \alpha }\)-smoothing and \({ \mathcal{X}^{\beta'} }\) is compactly embedded in~\({ \mathcal{X}^\beta }\) for \({ \beta' > \beta }\). Thus \({ \partial_{ \varphi} F^\epsilon(0, c) }\) is a compact perturbation of the identity and therefore constitutes a Fredholm operator of index zero. We may therefore apply the (analytic) Crandall--Rabinowitz theorem~\autocite[Theorems~8.3.1 and~8.4.1]{BufTol2003a} and obtain the following local bifurcation result.

\begin{theorem}[Local bifurcation] \label{thm:local-bifurcation}
For all \({ \epsilon > 0 }\) and \({ k \geq 1 }\) there exists a local, analytic curve
\begin{equation*}
\mathfrak{C}_{ \textnormal{loc}, k}^\epsilon \colon s \mapsto (\varphi_{ k }^{ \epsilon }(s), c_{ k }^{ \epsilon }(s)) \in \mathcal{X}^{ \beta } \times \R_+,
\end{equation*}
defined around \({ s = 0 }\), of nontrivial \({ 2 \uppi / k }\)-periodic solutions of~\eqref{eq:steady-reg} that bifurcates from the line of trivial solutions \({ c \mapsto (0, c) }\) at \({ \mathfrak{C}_{ \textnormal{loc}, k}^\epsilon(0) = (0, k^{- \alpha}) }\). In a neighbourhood of~\({ (0, k^{- \alpha}) }\) these are all the nontrivial solutions of \({ F^\epsilon(\varphi, c) = 0 }\) in \({ \mathcal{X}^{ \beta } \times \R_+ }\).
\end{theorem}

Since we have an analytic curve in \({ \mathcal{X}^\beta \times \R_+ }\), we may compute the associated asymptotic formulas for~\({ \mathfrak{C}_{ \textnormal{loc}, k}^\epsilon(s)  }\) as \({ s \to 0 }\) by means of direct expansions in the regularised steady equation~\eqref{eq:steady-reg}. Alternatively, one could use the more general theory in~\autocite[Section~I.6]{Kie2012e}.

\begin{theorem}[Bifurcation formulas] \label{thm:bifurcation-formulas}
\({ \mathfrak{C}_{ \textnormal{loc}, k}^{ \epsilon} }\) can be parametrised in such a way that \({ s \mapsto c_{ k }^{ \epsilon }(s) }\) is even, and with this choice the bifurcation formulas are as follows as~\({ s \to 0 }\):
\begin{alignat*}{2}
\text{In case~\eqref{eq:n-abs}:} \qquad \left\{\hspace{-1em}
\begin{aligned}
&&\varphi_k^\epsilon(s)(x) &= s \cos kx + s^2 \, C_{k, \textnormal{abs}}^\epsilon \cos 2kx + \mathcal{O}(s^3); \\[1ex]
&& c_k^\epsilon(s) &= k^{- \alpha} + s^2 \, 2 C_{k, \textnormal{abs}}^\epsilon + \mathcal{O}(s^4);
\end{aligned}\right. \\[2ex]
\text{in case~\eqref{eq:n-sgn}:} \qquad \left\{\hspace{-1em}
\begin{aligned}
&&\varphi_k^\epsilon(s)(x) &= s \cos kx + s^3 \, C_{k, \textnormal{sgn}}^\epsilon \cos 3kx + \mathcal{O}(s^5); \\[1ex]
&& c_k^\epsilon(s) &= k^{- \alpha} + s^2 \, \tfrac{ 3 }{ 4 }(p - 1) \epsilon^{p - 3} + \mathcal{O}(s^4),
\end{aligned}\right.
\end{alignat*}
with \({ C_{k, \textnormal{abs}}^\epsilon \coloneqq \displaystyle \frac{ \tfrac{ 1 }{ 4 } p \epsilon^{p - 2}  }{ k^{- \alpha} - (2k)^{- \alpha} } }\) and \({ C_{k, \textnormal{sgn}}^\epsilon \coloneqq \displaystyle \frac{ \tfrac{ 1 }{ 8 } (p - 1) \epsilon^{p - 3}  }{ k^{- \alpha} - (3k)^{- \alpha} } }\).
\end{theorem}
\vspace*{-.8em} 
\begin{remark} \label{rem:scaling}
It suffices to study the case \({ k = 1 }\) of \({ 2 \uppi }\)-periodic solutions in the bifurcation analysis, since \({ F^{ \epsilon}( \varphi, c) = 0 }\) is invariant under the scaling
\begin{equation*}
\varphi \mapsto k^{ \alpha / (p - 1) } \varphi(k \cdot), \qquad c \mapsto k^\alpha c, \qquad \epsilon \mapsto k^{ \alpha / (p - 1)} \epsilon.
\end{equation*}
Thus we focus on \({ \mathfrak{C}_{ \textnormal{loc}}^\epsilon \coloneqq \mathfrak{C}_{ \textnormal{loc}, 1}^\epsilon }\), \({ \varphi^\epsilon \coloneqq \varphi_{ 1}^{ \epsilon} }\), and \({ c^\epsilon \coloneqq c_{1}^{ \epsilon} }\) from now on.
\end{remark}
\vspace*{-.8em} 
\begin{proof}
As in the proof of~\autocite[Theorem~6.1]{EhrWah2019a}, we parametrise \({ \mathfrak{C}_{ \textnormal{loc}}^\epsilon }\) with the requirement \({ \cosineexpansion{ \varphi^\epsilon(s)}_1 = s }\), where
\begin{equation*}
\cosineexpansion{ \varphi}_m \coloneqq \tfrac{ 1 }{ \uppi } \int_{ \T } \varphi(x) \cos (m x) \dee x, \qquad m = 1, 2, \dotsc
\end{equation*}
are the coefficients in the cosine expansion \({ \varphi = \sum_{ m = 1 }^{ \infty } \cosineexpansion{ \varphi}_m \cos (m \cdot) }\). Since \({ (\varphi^\epsilon(\cdot + \uppi), c^\epsilon) }\) also constitutes a solution pair whenever \({ (\varphi^\epsilon, c^\epsilon) }\) is, and
\begin{equation*}
\cosineexpansion{ \varphi^\epsilon(s)(\cdot + \uppi) }_1 = - \cosineexpansion{ \varphi^\epsilon(s) }_1 = - s = \cosineexpansion{ \varphi^\epsilon(-s) }_1,
\end{equation*}
it follows by uniqueness that \({ \varphi^\epsilon(s)(\cdot + \uppi) = \varphi^\epsilon(-s) }\) and \({ c^\epsilon(s) = c^\epsilon(-s) }\), proving the symmetry.

Switching to the bifurcation formulas, we analytically expand \({ \varphi^\epsilon(s) }\) and \({ c^\epsilon(s) }\) into
\begin{equation} \label{eq:analytic-expansion}
\varphi^\epsilon(s) = \sum\nolimits_{ \ell = 1 }^{ \infty } \varphi_\ell \, s^\ell \qquad \textnormal{and} \qquad c^\epsilon(s) = \sum\nolimits_{ \ell = 0 }^{ \infty } \varsigma_{2\ell} \, s^{2\ell}
\end{equation}
and observe that the coefficients may be found by plugging the expansions into~\eqref{eq:steady-reg} and identifying terms of equal order in~\({ s }\) by uniqueness. Note that the Taylor expansion
\begin{equation} \label{eq:n-reg-taylor}
n^\epsilon(x) =
\begin{cases}
\tfrac{ 1 }{ 2 } p \epsilon^{p - 2} x^2 + \mathcal{O}(x^4) & \text{in case~\eqref{eq:n-abs};} \\[1ex]
\tfrac{ 1 }{ 2 }( p - 1) \epsilon^{p - 3} x^3 + \mathcal{O}(x^5) & \text{in case~\eqref{eq:n-sgn}}
\end{cases}
\end{equation}
holds in an \({ \epsilon }\)-dependent interval around \({ x = 0 }\), which simplifies the analysis for all sufficiently small~\({ s }\). With \({ L \coloneqq \abs{ \textnormal{D} }^{- \alpha} - \varsigma_0 \identityoperator }\), this gives the following in case~\eqref{eq:n-abs}:
\begin{align*}
s &\colon L \varphi_1 = 0; \\[1ex]
s^2 &\colon L \varphi_2 = - \tfrac{ 1 }{ 2 } p \epsilon^{p - 2} \bigl( \varphi_1^2 - \textstyle\fint_\T \varphi_1^2 \bigr); \\[1ex]
s^3 &\colon L \varphi_3 = \varsigma_2 \varphi_1 - p \epsilon^{p - 2} \bigl( \varphi_1 \varphi_2 - \textstyle\fint_\T \varphi_1 \varphi_2 \bigr).
\end{align*}
The first-order case yields that \({ \varphi_1 = \cos }\) and \({ \varsigma_0 = 1 }\) (more generally,~\({ \varsigma_0 = k^{- \alpha} }\)), whence
\begin{equation*}
\varphi_2(x) =  \frac{ \tfrac{ 1 }{ 4 } p \epsilon^{p - 2}  }{ 1 - 2^{- \alpha} } \cos 2x.
\end{equation*}
Since \({ 2 \cos x \cos 2x = \cos x + \cos 3x }\), it follows that
\begin{equation*}
\varsigma_2 = \frac{ \tfrac{ 1 }{ 8 } p^2 \epsilon^{2(p - 2) } }{ 1 - 2^{- \alpha} } \qquad \text{and} \qquad \varphi_3(x) =  \frac{ \tfrac{ 1 }{ 8 } p^2 \epsilon^{2(p - 2) } }{ (1 - 2^{- \alpha}) ( 1 - 3^{- \alpha} ) } \cos 3x.
\end{equation*}

As for case~\eqref{eq:n-sgn}, we find that\vspace*{-.5em} 
\begin{align*}
s &\colon L \varphi_1 = 0, \\[1ex]
s^2 &\colon L \varphi_2 = 0, \\[1ex]
\text{and } s^3 &\colon L \varphi_3 = \varsigma_2 \varphi_1 - \tfrac{ 1 }{ 2 }(p - 1) \epsilon^{p - 3} \bigl( \varphi_1^3 - \textstyle\fint_\T \varphi_1^3 \bigr),
\end{align*}
leading to \({ \varphi_1 = \cos }\) and \({ \varsigma_0 = 1 }\). Moreover, \({ \varphi_2 = 0 }\) by choice of parametrisation (\({ \cosineexpansion{ \varphi_\ell }_1 = 0 }\) for~\({ \ell \geq 2 }\)). Finally,
\begin{equation*}
\varsigma_2 = \tfrac{ 3 }{ 4 } (p - 1) \epsilon^{p - 3} \qquad \text{and} \qquad \varphi_3(x) = \frac{ \tfrac{ 1 }{ 8 }(p - 1)\epsilon^{p - 3} }{ 1 - 3^{- \alpha} } \cos 3x
\end{equation*}
with help of the identity \({ 4 \cos^3 x = 3 \cos x + \cos 3x }\).
\end{proof}

\vspace*{-.5em} 
\enlargethispage{.3\baselineskip} 

We also include asymptotic formulas in the smooth case~\({ n(x) = x^p }\) for any \({ 2 \leq p \in \N }\) (with no regularisation). Formulas with arbitrary order in~\({ p }\) seem to be new, and the result adapts easily to other dispersive operators as well.

\begin{theorem}[Bifurcation formulas for smooth~\({ n }\)] \label{thm:bifurcation-formulas-smooth} Consider smooth~\({ n(x) = x^p }\) for \({ 2 \leq p \in \N }\). Then the bifurcation formulas for~\({ \mathfrak{C}_{ \textnormal{loc}, k}^{ \epsilon = 0} }\) with the parametrisation that \({ s \mapsto c_{ k }^{ \epsilon = 0 }(s) }\) is even are as follows as~\({ s \to 0 }\):
\begin{equation*}
\varphi_{ k }^{ \epsilon = 0 }(s)(x) = s \cos kx + s^{p} \Phi_k(x) + \mathcal{O}(s^{2p - 1}),
\end{equation*}
where \({ \Phi_k^{} \in \Set*{ \Phi_k^{ \textnormal{even} },  \Phi_k^{ \textnormal{odd} } } }\) depending on whether \({ p \geq 2 }\) is even or odd, with corresponding speed
\begin{alignat*}{5}
c_{ k }^{ \epsilon = 0 }(s) &= k^{- \alpha} + s^{2p - 2} && C_k^{ \textnormal{even}} && + \mathcal{O}(s^{2p}) \\[-.5em]
\shortintertext{\hspace{4cm}or} \\[-2.5em]
c_{ k }^{ \epsilon = 0 }(s) &= k^{- \alpha} + s^{p - 1} && C_k^{ \textnormal{odd}} && + \mathcal{O}(s^{2p - 2}).
\end{alignat*}
Here
\begin{alignat*}{7} \SwapAboveDisplaySkip
& \Phi_k^{ \textnormal{even} } &&\coloneqq \textstyle\sum_{ j = 0 }^{ \frac{ p }{ 2 } - 1 } & \Phi_{k, j}, & \quad \hphantom{\text{and}} \quad C_k^{ \textnormal{even}} &&\coloneqq \textstyle\frac{ p }{ 2^{p - 1} } \Bigl( C_{k, 0} + \sum_{ j = 1 }^{ \frac{ p }{ 2 } - 1 } \left( \binom{ p - 1}{ j } + \binom{p - 1}{j - 1} \right) C_{k, j} \Bigr), \\[1ex]
& \Phi_k^{ \textnormal{odd} } &&\coloneqq \textstyle\sum_{ j = 0 }^{ \frac{ p - 3 }{ 2 } } & \Phi_{k, j}, & \quad \text{and} \quad C_k^{ \textnormal{odd}} &&\coloneqq \textstyle\frac{ 1 }{ 2^{p - 1} } \binom{ p }{ \frac{ p - 1 }{ 2 } },
\end{alignat*}
with \({ \displaystyle C_{k, j} \coloneqq \frac{ \binom{ p }{ j } / 2^{p - 1} }{ k^{- \alpha } - (( p - 2j) k)^{ -\alpha } } }\) and \({ \Phi_{k, j}(x) \coloneqq C_{k, j} \cos \bigl( (p - 2j) kx \bigr) }\).
\end{theorem}
\vspace*{-1.3em} 
\enlargethispage{1.8\baselineskip} 
\begin{proof}
The cases~\({ p = 2, 3 }\) are similar to those in the proof of \cref{thm:bifurcation-formulas}, and we only examine \({ k = 1 }\) by \cref{rem:scaling}. Thus let \({ L \coloneqq \abs{ \textnormal{D} }^{- \alpha} - \varsigma_0 \identityoperator }\) and consider first even~\({ p \geq 4 }\):
\begin{alignat*}{8}
\begin{alignedat}{4}
s  & \colon \quad L \varphi_1 &&= 0; \\[1ex]
s^2  & \colon \quad L \varphi_2 &&= 0; \\[1ex]
s^3  & \colon \quad L \varphi_3 &&= \varsigma_2 \varphi_1;
\end{alignedat}\qquad \dotso \qquad 
\begin{alignedat}{4}
s^{p - 1} & \colon \quad L \varphi_{p - 1} &&= \textstyle\sum_{ i = 1 }^{ \frac{ p }{ 2 } - 1 } \varsigma_{2i} \varphi_{ p - 1 - 2i}; \\[1ex]
s^{p} &\colon \quad L \varphi_{p}^{} &&= \textstyle\sum_{ i = 1 }^{ \frac{p }{ 2 } - 1 } \varsigma_{2i}^{} \varphi_{ p - 2i}^{} - \varphi_1^{p} + \textstyle\fint_\T \varphi_1^{p}.
\end{alignedat}
\end{alignat*}
The first-order case yields that \({ \varphi_1 = \cos }\) and \({ \varsigma_0 = 1 }\), and we successively find that
\begin{equation} \label{eq:trivial-coefficients}
\varphi_2 = \dotsb = \varphi_{p - 1} = 0 \qquad \textnormal{and} \qquad \varsigma_2 = \dotsb = \varsigma_{p - 2} = 0.
\end{equation}
This leads to
\begin{equation*}
\varphi_{p}^{} = - L^{-1} \left( \varphi_1^{p} - \textstyle\fint_\T \varphi_1^{p} \right) = \Phi_1^{ \textnormal{even}}
\end{equation*}
by means of the even power-reduction formula
\begin{equation*}
\cos^{p} x - \textstyle\fint_\T \cos^{p} = \displaystyle\frac{ 1 }{ 2^{p - 1} } \sum\nolimits_{ j = 1 }^{ \frac{ p }{ 2 } - 1 } \binom{p}{j} \cos \left( (p - 2j) x \right).
\end{equation*}
Taking the prior results into account, we next examine higher-order coefficients for even~\({ p \geq 4 }\) with help of successive cancellations of terms that vanish:\vspace*{-.3em} 
\begin{alignat*}{5}
s^{p + 1}  & \colon \quad L \varphi_{p + 1} &&= \varsigma_{p} \varphi_1 && \qquad \Rightarrow \qquad \varphi_{p + 1} &&= 0, \quad \varsigma_{p} = 0; \\[1ex]
s^{p + 2}  & \colon \quad L \varphi_{p + 2} &&= 0 && \qquad \Rightarrow \qquad \varphi_{p + 2} &&= 0;\\[1ex]
\vdots & && && \qquad \vdotswithin{\Rightarrow} \\[1ex]
s^{2p - 3}  & \colon \quad L \varphi_{2p - 3} &&= \varsigma_{2p - 2} \varphi_1 && \qquad \Rightarrow \qquad \varphi_{2p - 3} &&= 0, \quad \varsigma_{2p - 2} = 0; \\[1ex]
s^{2p - 2}  & \colon \quad L \varphi_{2p - 2} &&= 0 && \qquad \Rightarrow \qquad \varphi_{2p - 2} &&= 0;\\[1ex]
s^{2p - 1}  & \colon \quad L \varphi_{2p - 1}^{} &&= \varsigma_{2p - 2}^{} \varphi_1^{} - p \varphi_1^{p - 1} \varphi_{p}^{} + \textstyle\fint_\T \bigl(p \varphi_1^{p - 1} \varphi_{p}^{}\bigr). \span\span\span\span
\end{alignat*}
From the last equation it follows that
\begin{equation*}
\varsigma_{2p - 2}^{} = \textnormal{the coefficient of \({ (\varphi_1^{} =\mathrel{)} \cos }\) in } p\left( \varphi_1^{p - 1} \varphi_{p}^{} - \textstyle\fint_\T \bigl(\varphi_1^{p - 1} \varphi_{p}^{}\bigr) \right) = C_1^{ \textnormal{even}}
\end{equation*}
with help of the odd power-reduction formula \({ \smash{\cos^q x = 2^{-q} \sum\nolimits_{ j = 1 }^{ \frac{ q - 1 }{ 2 } } \binom{q}{j} \cos \left( (q - 2j) x \right)} }\) for \({ q = p - 1 }\) and the product-to-sum identity for cosine.

Switching to odd \({ p \geq 5 }\), we similarly obtain that \({ \varphi_1 = \cos }\) and \({ \varsigma_0  = 1 }\) and that~\eqref{eq:trivial-coefficients} is true. Moreover, from
\begin{equation*}
s^{p} \colon \quad L \varphi_{p}^{} = \varsigma_{p - 1}^{} \varphi_1^{} - \varphi_1^{p} + \underbrace{\textstyle\fint_\T \varphi_1^{p}}_{\textnormal{\makebox[0pt]{\({ = 0 }\) for odd~\({ p }\)}}}
\end{equation*}
we finally deduce that
\begin{align*} \SwapAboveDisplaySkip
\varsigma_{p - 1}^{} &= \textnormal{the coefficient of \({ \cos }\) in } \varphi_1^{p} = C_1^{ \textnormal{odd}}
\shortintertext{and}
\varphi_{p}^{}  &= L^{-1} \bigl( \varsigma_{p - 1}^{} \varphi_1^{} - \varphi_1^{p} \bigr) = \Phi_1^{ \textnormal{odd}},
\end{align*}
again by the odd power-reduction formula.
\end{proof}

For odd~\({ p }\) we can improve upon \cref{thm:bifurcation-formulas-smooth} and obtain the overall structure of the bifurcation formulas near the line of trivial solutions. This shows that \({ \varphi^{ \epsilon = 0}(s) }\) is antisymmetric about~\({ - \frac{ \uppi }{ 2 } }\), and agrees with the general conjecture set forth in \cref{sec:regularity}.

\begin{proposition}[Local antisymmetry] \label{thm:local-antisymmetry}
Consider \({ n(x) = x^p }\) for odd~\({ p \geq 3 }\) and the choice of parametrisation in \cref{thm:bifurcation-formulas-smooth}. Then the analytic structure~\eqref{eq:analytic-expansion} of the local bifurcation formulas equals
\begin{equation*}
\varphi^{ \epsilon = 0}(s) = \sum_{j = 0}^\infty \varphi_{ j(p - 1) + 1} \, s^{ j(p - 1) + 1} \qquad \text{and} \qquad c^{ \epsilon = 0}(s) = \sum_{j = 0}^\infty \varsigma_{j(p - 1)} \, s^{j(p - 1)},
\end{equation*}
on~\({ \mathfrak{C}_{ \textnormal{loc}}^{\epsilon = 0} }\), where all the \({ \varphi_{ j(p - 1) + 1} }\) functions lie in \({ \smash{\displaystyle \mathcal{W} \coloneqq \smashoperator{\vectorspan_{ \textnormal{odd } k \geq 1}} \Set{ \cos (k \cdot) } } }\). Hence \({ \varphi_{ j(p - 1) + 1} }\) and thus also \({ \varphi^{ \epsilon = 0} }\) are antisymmetric about~\({ - \tfrac{ \uppi }{ 2 } }\).
\end{proposition}
\vspace*{-.5em} 
\begin{remark}
\Cref{thm:local-antisymmetry} also hold in case~\eqref{eq:n-sgn} of~\eqref{eq:n-reg} with~\({ \epsilon > 0 }\), provided~\({ s }\) is sufficiently small. In this case the representations become
\begin{equation*}
\varphi^{ \epsilon}(s) = \textstyle\sum_{j = 0}^\infty \varphi_{ 2j + 1} \, s^{2j + 1} \qquad \text{and} \qquad c^{ \epsilon}(s) = \textstyle\sum_{j = 0}^\infty \varsigma_{2j} \, s^{2j},
\end{equation*}
as indicated by \cref{thm:bifurcation-formulas}.
\end{remark}
\vspace*{-.5em} 
\begin{proof}
We use strong induction, where the base case is given by \cref{thm:bifurcation-formulas-smooth}. Let \({ q \coloneqq p  - 1 }\) and suppose the result is true for \({ \Set{ 0q, \dotsc,  jq } }\) for some~\({ j \geq 0 }\). Now consider case~\({ (j + 1)q }\). As in the proof of \cref{thm:bifurcation-formulas,thm:bifurcation-formulas-smooth}, we insert~\eqref{eq:analytic-expansion} into~\eqref{eq:steady}, identify terms of equal order in~\({ s }\), and simplify by means of successive cancellations, with \({ L \coloneqq \abs{ \textnormal{D} }^{- \alpha} - \varsigma_0 \identityoperator }\):
\begingroup
\allowdisplaybreaks
\begin{alignat*}{5}
s^{jq + 2}  & \colon \quad L \varphi_{jq + 2} &&= 0 && \qquad \Rightarrow \qquad \varphi_{jq + 2} &&= 0;\\[1ex]
s^{ jq + 3}  & \colon \quad L \varphi_{jq + 3} &&= \varsigma_{jq + 2} \varphi_1 && \qquad \Rightarrow \qquad \varphi_{jq + 3} &&= 0, \quad \varsigma_{jq + 2} &&= 0; \\[1ex]
\vdots & && && \qquad \vdotswithin{\Rightarrow} \\[1ex]
s^{(j + 1)q - 1}  & \colon \quad L \varphi_{(j + 1)q - 1} &&= \varsigma_{(j + 1)q - 2} \varphi_1 && \qquad \Rightarrow \qquad \varphi_{(j + 1)q - 1} &&= 0, \quad \varsigma_{(j + 1)q - 2} &&= 0; \\[1ex]
s^{(j + 1)q}  & \colon \quad L \varphi_{(j + 1)q} &&= 0 && \qquad \Rightarrow \qquad \varphi_{(j + 1)q} &&= 0;\\[1ex]
s^{(j + 1)q + 1}  & \colon \quad L \varphi_{ (j + 1)q + 1} &&= \textstyle\sum_{ i = 1 }^{ j + 1 } \varsigma_{ iq} \varphi_{ (j - i + 1)q + 1 } - \Lambda, \span\span\span\span
\end{alignat*}
\endgroup
where
\begin{align*}
\Lambda &\coloneqq \smash[t]{\Biggl\{ \textstyle\binom{ p}{1} \varphi_1^{ q} \varphi_{ jq + 1}^{} + \textstyle\binom{ p }{2} \varphi_1^{ q - 1} \biggl( \varphi_{ q + 1}^{} \varphi_{ (j - 1)q + 1}^{} + \varphi_{ 2q + 1}^{} \varphi_{ (j - 2)q + 1}^{} + \dotsb +  \varphi_{ \lfloor \frac{ j }{ 2 } \rfloor q + 1} \varphi_{ \lceil \frac{ j }{ 2 } \rceil q + 1} \biggr)} \\[1ex]
& \hphantom{\coloneqq \Biggl\{} + \dotsb + \textstyle\binom{ p}{j} \varphi_1^{ q - j + 1} \varphi_{ q + 1}^{ j } \Biggr\}
\end{align*}
follows by expanding \({ (\varphi^{\epsilon = 0}(s))^p }\). Then we obtain
\begin{equation*}
\varsigma_{ (j + 1)q} = \textnormal{coefficient of \({(\varphi_1=) \cos }\) in } \Lambda \qquad \text{and} \qquad \varphi_{ (j + 1)q + 1}^{} = L^{-1} \left( \textstyle\sum_{ i = 1 }^{ j + 1 } \varsigma_{ iq}^{} \varphi_{ (j - i + 1) q + 1 }^{} - \Lambda) \right).
\end{equation*}
By the induction hypothesis we know that \({ \varphi_{ \widetilde{ j } q + 1} \in \mathcal{W} }\) for all \({ 0 \leq \widetilde{ j } \leq j }\). Moreover, each term in \({ \Lambda }\) is the product of an \emph{odd} number of (some of) the terms~\({ \varphi_{ \overline{ j } q + 1} }\), with repetitions allowed. This establishes the result by noting that \({ \mathcal{W} }\) is algebraically closed under an odd number of multiplications, which can be deduced from the identity
\begin{equation*}
4 \cos u \cos v \cos w = \cos(\underbrace{u + v+ w}_{ = \textnormal{ odd}}) + \cos(\underbrace{-u + v + w}_{ = \textnormal{ odd}}) + \cos(\underbrace{u - v + w}_{ = \textnormal{ odd}}) + \cos(\underbrace{u + v - w}_{ = \textnormal{ odd}})
\end{equation*}
whenever \({ u }\), \({ v }\), and \({ w }\) are odd. General products reduce iteratively to triple products.
\end{proof}
\vspace*{-.5em} 

Although \cref{thm:local-antisymmetry} is promising, it is not clear to us how one can prove antisymmetry everywhere along \({ \mathfrak{C}_{ \textnormal{loc}}^\epsilon }\) and its upcoming global extension. Thus we instead redefine \({ \mathcal{X}^{ \beta} }\) in case~\eqref{eq:n-sgn} as the subspace
\begin{equation*}
\Set*{ \varphi \in \mathring{ \textnormal{C}}_{ \textnormal{even}}^\beta(\T) \given \varphi \textnormal{ is antisymmetric about } {-} \tfrac{ \uppi }{ 2 } \textnormal{, that is, } \varphi(\cdot + \uppi) = - \varphi },
\end{equation*}
for which correspondingly \({ \ker \partial_{ \varphi} F^\epsilon(0, k^{ - \alpha } ) = \mathcal{W} }\) and \cref{thm:local-bifurcation,thm:bifurcation-formulas} hold for odd~\({ k }\).

We proceed to analyse the global structure of an extension of \({\mathfrak{C}_{ \textnormal{loc}}^\epsilon }\) in \cref{thm:local-bifurcation}. To this end, let
\begin{equation*}
\mathcal{S}^\epsilon \coloneqq \Set[\big]{ (\varphi, c) \in \mathcal{U}^\epsilon \given F^\epsilon(\varphi, c) = 0 }
\end{equation*}
be the set of admissible solution pairs, where
\begin{equation*}
\mathcal{U}^\epsilon \coloneqq \Set[\big]{ (\varphi, c) \in \mathcal{X}^{ \beta } \times \R_+ \given (n^\epsilon)'(\varphi) < c  }
\end{equation*}
is an open set whose boundary contains any solution pair of~\eqref{eq:steady-reg} with the desirable regularity features in \cref{thm:regularity}. We first note the following property of~\({ \mathcal{S}^\epsilon }\).

\begin{lemma} \label{thm:compactness-solution-set}
Bounded, closed subsets of \({ \mathcal{S}^\epsilon }\) are compact in~\({ \mathcal{X}^{ \beta } \times \R_+ }\).
\end{lemma}
\vspace*{-.8em} 
\enlargethispage{\baselineskip} 
\begin{proof}
It follows from \cref{thm:smoothness} and its proof, that the operator~\({ G }\) in~\eqref{eq:composition-operator}---adapted with \({ N^\epsilon }\) replacing the nonsmooth~\({ N }\)---sends \({ (\varphi, c) }\) to~\({ \varphi }\) on~\({ \mathcal{S}^\epsilon }\) and boundedly maps~\({ \mathcal{S}^\epsilon }\) into \({ \textnormal{C}^{m} }\) for any~\({ m \geq 1 }\). Since \({ \mathcal{X}^{ \beta' } }\) is compactly embedded in \({ \mathcal{X}^{ \beta } }\) for \({ \beta' > \beta }\), we find that
\begin{equation*}
G \textnormal{ maps bounded subsets of } \mathcal{S}^\epsilon \textnormal{ into relatively compact subsets of } \mathcal{X}^{ \beta }.
\end{equation*}
In particular, if \({ \Set{ (\varphi_j, c_j) }_j }\) is a sequence in a bounded subset~\({ \mathcal{B} \subseteq \mathcal{S}^\epsilon }\), then a subsequence of \({ \Set{ \varphi_j }_j }\) converges in~\({ \mathcal{X}^{ \beta } }\), which together with the Bolzano--Weierstrass theorem imply that a subsequence of \({ \Set{ (\varphi_j, c_j) }_j }\) converges in the \({ \mathcal{X}^{ \beta } \times \R_+ }\)-topology. Thus if \({ \mathcal{B} }\) also is closed, it is compact.
\end{proof}
By means of \cref{thm:compactness-solution-set} and the fact that \({ c^\epsilon(s) }\) is not identically constant due to \cref{thm:bifurcation-formulas}, we may appeal to~\autocite[Theorem~9.1.1]{BufTol2003a} and obtain a global extension of~\({ \mathfrak{C}_{ \textnormal{loc}}^\epsilon }\). Note that we do not distinguish between a curve and its image.

\begin{theorem}[Global bifurcation] \label{thm:global-bifurcation}
\({ \mathfrak{C}_{ \textnormal{loc}}^\epsilon }\) extends to a global continuous curve \({ \mathfrak{C}^\epsilon \colon \R_+ \to \mathcal{S}^\epsilon }\) of solution pairs \({ \mathfrak{C}^\epsilon(s) = (\varphi^\epsilon(s), c^\epsilon(s)) }\), and either
\begin{tasks}[label={\textcolor{artcolor}{\roman*)}},ref={\roman*)}](3)
\task \label{thm:global-bifurcation-blowup} \({ \displaystyle\lim_{ s \to \infty } \norm{ \mathfrak{C}^\epsilon(s) }_{ \mathcal{X}^{ \beta } \times \R_+ } = \infty }\),
\task \label{thm:global-bifurcation-boundary} \({ \distance{ \mathfrak{C}^\epsilon }{ \partial \mathcal{U}^\epsilon} = 0 }\), \qquad or
\task \label{thm:global-bifurcation-periodic} \({ \mathfrak{C}^\epsilon }\) is periodic.
\end{tasks}
\end{theorem}

We shall prove that possibility~\ref{thm:global-bifurcation-periodic} does not happen and that possibilities~\ref{thm:global-bifurcation-blowup} and~\ref{thm:global-bifurcation-boundary} occur simultaneously, from which it will follow that one finds a highest, \({ \alpha }\)-Hölder continuous wave as a limit along~\({ \mathfrak{C}^\epsilon }\).

In order to eliminate the possibility that \({ \mathfrak{C}^\epsilon }\) is periodic, we make use of a conic refinement of the global bifurcation theorem~\autocite[Theorem~9.1.1]{BufTol2003a}. Specifically, let
\begin{equation*}
\mathcal{K} \coloneqq \Set*{ \varphi \in \mathcal{X}^{ \beta } \given \varphi \textnormal{ is increasing on } (- \uppi, 0) }
\end{equation*}
be a closed cone in \({ \mathcal{X}^{ \beta } }\) and observe that \({ \mathfrak{C}^\epsilon(s) \in \mathcal{K} \times \R_+ }\) for sufficiently small~\({ s }\). Indeed, cosine is strictly increasing on~\({ (- \uppi, 0) }\) and strict monotonicity is stable under \({ \textnormal{C}^1 }\)-perturbations on a compact set~(here,~\({ \T }\)). Therefore, since \({ \varphi^\epsilon(s) = s \cos + \mathcal{O}(s^2) }\) from \cref{thm:local-bifurcation,thm:bifurcation-formulas} is smooth on~\({ \T }\) by \cref{thm:smoothness} (adapted to~\eqref{eq:steady-reg} with~\({ n^\epsilon }\)), it holds that \({ \varphi^\epsilon(s) \in \mathcal{K} \setminus \Set{ 0 } }\) for small~\({ s = o(\epsilon) }\). In fact, this is true globally.

\begin{proposition} \label{thm:cone}
\({ \varphi^\epsilon(s) \in \mathcal{K} \setminus \Set{ 0 } }\) for all \({ s > 0 }\) and \({ 0 < \epsilon \ll 1 }\). In particular, \({ \mathfrak{C}^\epsilon }\) never returns (for finite~\({ s }\)) to the line of trivial solutions, thereby ruling out possibility~\ref{thm:global-bifurcation-periodic} in \cref{thm:global-bifurcation}.
\end{proposition}

\begin{proof}
According to \autocite[Theorem~9.2.2]{BufTol2003a}, it suffices to show that each \({ (\varphi^\epsilon, c^\epsilon) }\) on~\({ \mathfrak{C}^\epsilon }\) which also belongs to \({ (\mathcal{K} \setminus \Set{ 0 }) \times \R_+ }\) lies in the interior of \({ (\mathcal{K} \setminus \Set{ 0 }) \times \R_+ }\) in the topology of~\({ \mathcal{S}^\epsilon }\). To this end, observe by \cref{thm:smoothness,thm:nodal-pattern} that such \({ \varphi^\epsilon }\) with speed~\({ c^\epsilon }\) is smooth and satisfies \({ (\varphi^\epsilon)' > 0 }\) on~\({ (- \uppi, 0) }\), with \({ (\varphi^\epsilon)''(0) < 0 }\) and \({ (\varphi^\epsilon)''(- \uppi) > 0 }\). Now let \({ (\phi, d) \in \mathcal{S}^\epsilon }\) be another solution (not necessarily on~\({ \mathfrak{C}^\epsilon }\)) lying within \({ \delta }\)-distance to~\({ (\varphi^\epsilon, c^\epsilon) }\) in~\({ \mathcal{X}^{ \beta } \times \R_+ }\). Then \({ \phi }\) and \({ d }\) are nonzero, and \({ \phi }\) is also smooth (\cref{thm:smoothness}). Moreover, \({ (N^\epsilon)^{-1} }\) is smooth---also as a function of the wave speed. Hence, it follows from~\autocite[Theorems~2.2, 4.2 and~5.2]{GoeSac1999a} and iteration of the smoothing effect of~\({ \abs{ \textnormal{D} }^{- \alpha} }\) that \({ G }\) in~\eqref{eq:composition-operator} (with \({ N^\epsilon }\) replacing~\({ N }\)) is a continuous map \({ \mathcal{S}_{}^\epsilon \to \mathcal{S}_1^\epsilon \cap \mathcal{X}^m }\) for any integer~\({ m \geq 1 }\), where \({ \mathcal{S}_1^\epsilon }\) is the functional component of~\({ \mathcal{S}^\epsilon }\). As such,
\begin{equation*}
\norm{ \phi - \varphi^\epsilon }_{ \mathring{\textnormal{C}}^2(\T)} = \norm{ G(\phi, d) - G(\varphi^\epsilon, c^\epsilon) }_{ \mathring{\textnormal{C}}^2(\T) } < \tau(\delta)
\end{equation*}
when \({ \norm{ (\phi, d) - (\varphi^\epsilon, c^\epsilon) }_{ \mathcal{X}^\beta \times \R_+ } < \delta }\). Thus for sufficiently small~\({ \delta }\), one deduces that \({ \phi }\) is strictly increasing on~\({ (- \uppi, 0) }\), so that \({ \phi \in \mathcal{K} \setminus \Set{ 0 } }\).
\end{proof}

\begin{remark}
Proofs of similar results (for instance \autocite[Theorem~6.7]{EhrWah2019a}, \autocite[Proposition~5.9]{BruDha2021a}, and~\autocite[Theorem~4.6]{Arn2019b}) as \cref{thm:cone} seem to disregard that \({ G }\) depends on the wave speed. But \({ G(\varphi, d) }\) does not necessarily equal~\({ \varphi }\) when \({ d \neq c }\) and~\({ (\varphi, c) \in \mathcal{S}^\epsilon }\), and it is key to work with open \({ \delta }\)-balls around solution \emph{pairs} \({ (\varphi, c) \in \mathcal{S}^\epsilon }\) and not only around solutions~\({ \varphi }\).
\end{remark}
\Cref{thm:lower-bound-speed} (adapted to~\eqref{eq:steady-reg} with~\({ n^\epsilon }\)) and \cref{thm:cone} immediately imply the following result.

\begin{corollary} \label{thm:lower-bound-speed-bifurcation}
The wave speed~\({ c^\epsilon(s) }\) is uniformly bounded away from~\({ 0 }\) along~\({ \mathfrak{C}^\epsilon }\) and~\({ 0 < \epsilon \ll 1 }\).
\end{corollary}

In the remainder, we let \({ \Set {(\varphi_j^\epsilon, c_j^\epsilon)}_j^{} \coloneqq \Set{(\varphi^\epsilon(s_j^{}), c^\epsilon(s_j^{}))}_j^{} }\) denote a generic sequence along~\({ \mathfrak{C}^\epsilon }\) with \({ s_j \to \infty }\) as~\({ j \to \infty }\).

\begin{proposition} \label{thm:convergence}
Any sequence \({ \Set {(\varphi_{ j }^{ \epsilon }, c_{ j }^{ \epsilon })}_j^{} }\) with \({ \Set{c_j^\epsilon}_j^{} }\) bounded has a subsequence converging to a solution of~\eqref{eq:steady-reg} in~\({ \mathcal{X}^0 \times \R_+ }\).
\end{proposition}

\begin{proof}
Note from \cref{thm:bound-wavespeed} (adapted to~\eqref{eq:steady-reg} with~\({ n^\epsilon }\)) that \({ \Set{ \varphi_j^\epsilon }_j^{} }\) is bounded in~\({ \mathcal{X}^0 }\). Since \({ K_\alpha }\) is integrable and translation in \({ \textnormal{L}^1(\T) }\) is uniformly continuous, it follows from
\begin{equation*}
\abs[\big]{ \abs{ \textnormal{D} }^{- \alpha} \varphi_j^\epsilon(x) - \abs{ \textnormal{D} }^{- \alpha} \varphi_j^\epsilon(y) } \leq \norm{ K_\alpha(x - \cdot) - K_\alpha(y - \cdot) }_{ \textnormal{L}^1(\T) } \sup\nolimits_{ j }^{} \norm{ \varphi_j^\epsilon }_{ \infty }
\end{equation*}
that \({ \Set{ \abs{ \textnormal{D} }^{- \alpha} \varphi_j^\epsilon }_j^{} }\) is (uniformly) equicontinuous on~\({ \T }\). Moreover,
\begin{align*}
\abs{ \textnormal{D} }^{- \alpha} \varphi_j^\epsilon(x) - \abs{ \textnormal{D} }^{- \alpha} \varphi_j^\epsilon(y) &= N^\epsilon(\varphi_j^\epsilon(x)) - N^\epsilon(\varphi_j^\epsilon(y)) \\[1ex]
&= \bigl( \varphi_j^\epsilon(x) - \varphi_j^\epsilon(y) \bigr) (N^\epsilon)'(\varphi_j^\epsilon(\xi_j^{}))
\end{align*}
for some \({ \xi_j }\) between \({ x, y \in \T }\), which since~\({ (N^\epsilon)'(\varphi_j^\epsilon(\xi_j^{})) > 0 }\), implies equicontinuity of \({ \Set{ \varphi_j^\epsilon }_j^{} }\) strictly away from~\({ 0 }\) (and~\({ - \uppi }\) in case~\eqref{eq:n-sgn}). Patched together with~\eqref{eq:L-difference-square} around~\({ 0 }\), we infer that \({ \Set{ \varphi_j^\epsilon }_j^{} }\) is equicontinuous on all of~\({ \T }\). Thus a subsequence converges in~\({ \mathcal{X}^0 }\) by the Arzelà--Ascoli theorem. Continuity of \({ \abs{ \textnormal{D} }^{- \alpha} }\) and \({ n^\epsilon }\) on \({ \mathcal{X}^0 }\) together with the existence of a convergent subsequence of~\({ \Set{ c_j^\epsilon }_j^{} }\) (by the Bolzano--Weierstrass theorem), then show that a subsequence of \({ \Set{ (\varphi_j^\epsilon, c_j^\epsilon) }_j^{} }\) converges to a solution of~\eqref{eq:steady-reg} in~\({ \mathcal{X}^0 \times \R_+ }\).
\end{proof}

\begin{proposition} \label{thm:simultaneous}
Possibilities~\ref{thm:global-bifurcation-blowup} and~\ref{thm:global-bifurcation-boundary} in \cref{thm:global-bifurcation} occur simultaneously.
\end{proposition}

\begin{proof}
In light of \cref{thm:cone}, we know that either possibility~\ref{thm:global-bifurcation-blowup} or possibility~\ref{thm:global-bifurcation-boundary} takes place, and that \({ \varphi^\epsilon(s) }\) is nontrivial and increasing on~\({ (- \uppi, 0) }\) for~\({ s > 0 }\) by \cref{thm:cone}. If possibility~\ref{thm:global-bifurcation-blowup} occurs, then either \({ \norm{ \varphi^\epsilon(s) }_{ \mathcal{X}^\beta } \to \infty }\) or \({ c^\epsilon(s) \to \infty }\) as~\({ s \to \infty }\). Since the wave speed cannot blow up due to \cref{thm:bound-wavespeed} (adapted to~\eqref{eq:steady-reg}) and \({ \varphi^\epsilon(s) }\) being nontrivial, it must be that \({ \norm{ \varphi^\epsilon(s) }_{ \mathcal{X}^\beta }  }\) explodes. But then
\begin{equation*}
c^\epsilon(s) -  (n^\epsilon)'(\varphi^\epsilon(s)(x)) \xrightarrow[ s \to \infty ]{} 0
\end{equation*}
at \({ x = 0 }\) (and at \({ x = - \uppi }\) in case~\eqref{eq:n-sgn}) by \cref{thm:smoothness} adapted to~\eqref{eq:steady-reg}, demonstrating that possibility~\ref{thm:global-bifurcation-boundary} holds.

Conversely, suppose that possibility~\ref{thm:global-bifurcation-boundary} but not possibility~\ref{thm:global-bifurcation-blowup} occurs. Then there exists a sequence \({ \Set{ (\varphi_j^\epsilon, c_j^\epsilon) }_j^{} }\) along~\({ \mathfrak{C}^\epsilon }\), with \({ \varphi_j^\epsilon }\) increasing on~\({ (- \uppi, 0) }\), satisfying \({ (n^\epsilon)'(\varphi_j^\epsilon) < c_j^\epsilon }\) everywhere and\pagebreak 
\begin{equation*}
c_j^\epsilon - (n^\epsilon)'(\varphi_j^\epsilon(0)) \xrightarrow[ j \to \infty ]{} 0, \qquad \textnormal{equiv.\@ that} \qquad \mu_j^\epsilon - \varphi_j^\epsilon(0) \xrightarrow[ j \to \infty ]{} 0,
\end{equation*}
while \({ \Set{ \varphi_j^\epsilon }_j^{} }\) remains bounded in~\({ \mathcal{X}^\beta }\), where \({ \mu_j^\epsilon \coloneqq \mu(p, c_j^\epsilon, \epsilon) }\) as in~\eqref{eq:highest-value-reg}. By compactness we may extract a convergent subsequence in~\({ \mathcal{X}^{ \beta'} }\) for~\({ \beta' \in (\alpha, \beta) }\), which yields a contradiction to \cref{thm:lower-bound-zero} (adapted to~\eqref{eq:steady-reg}) with respect to the one-sided \({ \alpha }\)-Hölder rate at~\({ 0 }\). Hence, possibility~\ref{thm:global-bifurcation-blowup} is true.
\end{proof}

In order to conclude the proof of \cref{thm:existence}, let \({ \Set{ (\varphi_j^\epsilon, c_j^\epsilon) }_j^{} }\) be any sequence along~\({ \mathfrak{C}^\epsilon }\) for fixed \({ 0 < \epsilon \ll 1 }\). By \cref{thm:bound-wavespeed} (adapted to~\eqref{eq:steady-reg}) we know that \({ \Set{ c_j^\epsilon }_j^{} }\) is bounded, and so \cref{thm:convergence} shows that \({ \Set{ (\varphi_j^\epsilon, c_j^\epsilon) }_j^{} }\) converges, up to a subsequence, to a solution~\({ (\varphi^\epsilon, c^\epsilon) \in \mathcal{X}^0 \times \R_+ }\) of~\eqref{eq:steady-reg} with \({ \varphi^\epsilon \neq 0 }\) increasing on~\({ (- \uppi, 0) }\) by \cref{thm:cone} and \({ c^\epsilon \neq 0 }\) due to \cref{thm:lower-bound-speed-bifurcation}. It is then clear from \cref{thm:simultaneous} that \({ (n^\epsilon)'(\varphi^\epsilon(0)) = c^\epsilon }\) or equivalently, that \({ \varphi^\epsilon(0) = \mu^\epsilon }\) by~\eqref{eq:highest-value-reg}.

Now let \({ \epsilon \searrow 0 }\). \Cref{thm:bound-wavespeed,thm:lower-bound-speed} (adapted to~\eqref{eq:steady-reg}) imply that \({ \Set{ c^\epsilon }_\epsilon }\) converges, up to a subsequence, to some~\({ c \neq 0 }\), from which we also find that \({ \Set{ \varphi^\epsilon }_\epsilon }\) is bounded in~\({ \mathcal{X}^0 }\). As in the proof of~\cref{thm:convergence}, there exists a uniformly convergent subsequence (not relabeled) with limit~\({ \varphi \in \mathcal{X}^0 }\) by the Arzelà--Ascoli theorem. Since also \({ n^\epsilon \to n }\) uniformly (locally in~\({ \R }\)) by its construction~\eqref{eq:n-reg}, we infer that
\begin{equation*}
n^\epsilon(\varphi^\epsilon) \to n(\varphi) \quad \text{in } \mathcal{X}^0.
\end{equation*}
Coupled with continuity of~\({ \abs{ \textnormal{D} }^{ - \alpha} }\) on~\({ \mathcal{X}^0 }\), it follows that \({ \Set{ (\varphi^\epsilon, c^\epsilon) }_\epsilon }\) converges, up to a subsequence, to a solution~\({ (\varphi, c) \in \mathcal{X}^0 \times \R_+ }\) of the original equation~\eqref{eq:steady}, with \({ n'(\varphi) \leq c }\) and \({ \varphi }\) being increasing on~\({ (- \uppi, 0) }\), and with \({ \varphi }\) also being antisymmetric about~\({ - \tfrac{ \uppi }{ 2 } }\) in case~\eqref{eq:n-sgn}. Observe finally that \({ \varphi }\) is nontrivial, because
\begin{equation*}
\varphi(0) = \lim_{ \epsilon \searrow 0 } \varphi^\epsilon(0) = \lim_{ \epsilon \searrow 0 } \mu^\epsilon = \mu \neq 0,
\end{equation*}
where \({ \mu }\) is as in~\eqref{eq:highest-value}. This then finishes the proof in light of \cref{thm:regularity}.

\section{Conclusion} 

In this paper, we have established the existence of large-amplitude periodic travelling-wave solutions with exact and optimal \({ \alpha }\)-Hölder regularity in a class of evolution equations with negative-order homogeneous dispersion of order~\({ - \alpha }\) for all~\({ \alpha \in (0, 1) }\). Techniques include elaborate local estimates for nonlocal operators and global bifurcation analysis. A main novelty is the inclusion of generally nonsmooth, power-type nonlinearities in the considered class of equations, which we analyse using a regularisation process. We also obtain that antisymmetric nonlinearities lead to the first existence result of \enquote{doubly-cusped} extreme waves with antisymmetry.

These results open up for new investigations. One may, for instance, consider inhomogeneous nonlinearities and also study associated symmetry principles for the existence of large-amplitude waves. Another line of research may seek to establish the convexity of the highest waves and its connection to the order of the dispersive operator and the growth and regularity of the nonlinearity.

\section{Acknowledgements} 

The authors gratefully acknowledge the in-depth feedback from the anonymous referee. Both authors were partially supported by research grant no.\@ 250070 from The Research Council of Norway.

\printbibliography
\end{document}